\documentclass[11pt]{article}

\usepackage[pagebackref,colorlinks=true,urlcolor=blue,linkcolor=red,citecolor=red]{hyperref}

\usepackage{epsfig, graphics, graphicx}
\usepackage{subcaption, comment, bm}
\usepackage[percent]{overpic}
\usepackage{float}
\usepackage{amssymb,bm}
\usepackage{amsthm}
\usepackage{color}
\usepackage{amsmath}
\usepackage[]{amsfonts}
\usepackage{fancyhdr}
\usepackage[]{graphicx,wrapfig}
\usepackage{enumitem}
\usepackage{array}

\newtheorem{theorem}{Theorem}[]
\newtheorem{lemma}[theorem]{Lemma}

\newtheorem{corollary}[theorem]{Corollary}

\newcommand{\cout}[1]{}

\usepackage{graphicx}
\usepackage[dvipsnames]{xcolor}

\newcommand{\Lc}{\mathcal{L}}

\newcommand{\Nc}{\mathcal{N}}

\newcommand{\Rb}{\mathbb{R}}
\newcommand{\Sb}{\mathbb{S}}

\renewcommand{\O}{\Omega}

\renewcommand{\o}{\omega}
\usepackage{amsthm}

\usepackage{amsfonts}

\newcommand{\Beq}{\begin{equation}}
\newcommand{\Eeq}{\end{equation}}
\newcommand{\beq}{\begin{equation*}}
\newcommand{\eeq}{\end{equation*}}
\newcommand{\bal}{\begin{align}}
\newcommand{\eal}{\end{align}}

\graphicspath{{./Figs/}}

\newtheorem{Prop}{Proposition}

\title{\vspace{-1cm} Determining time-dependent convection and density terms in convection-diffusion equation using partial data}

\author{Anamika Purohit\thanks{Department of Mathematics, Indian Institute of Technology, Gandhinagar, Gujarat, India. \url{anamika.purohit@iitgn.ac.in }}}

\begin{document}
\date{}
\maketitle
\begin{abstract}
In this article, we study an inverse boundary value problem for the time-dependent convection-diffusion equation. We use the nonlinear Carleman weight to recover the time-dependent convection term and time-dependent density coefficient uniquely. Nonlinear weight allows us to prove the uniqueness of the coefficients by making measurements on a possibly very small subset of the boundary. We proved that the convection term and the density coefficient can be recovered up to the natural gauge from the knowledge of the Dirichlet to Neumann map measured on a very small open subset of the boundary. 
\end{abstract}
\vspace{-5mm}

\section{Introduction}\label{Introduction}
Let $\Omega$ be an open and bounded subset of $\mathbb{R}^n$($n\geq 3$) with a smooth boundary $\partial \O$. For $T > 0$, $\Omega_T := (0,T)\times \Omega$ denotes the open solid cylinder with base $\O$ and height $T$ and its lateral boundary is denoted by $\partial\Omega_T := (0,T)\times \partial\Omega$. For $A:= (A_1, \cdots, A_n)$ with each  $A_j \in W^{1,\infty}(\Omega_T)$ and $q \in L^{\infty}(\Omega_T)$, consider the following initial boundary value problem for the convection-diffusion equation
\begin{align}
 \label{eq:1.1}
  \left\{
	\begin{array}{ r l l }
   \displaystyle \mathcal{L}_{A,q}u(t,x) := \bigg[\partial_t -  \sum_{j=1}^{n}(\partial_j+ A_j(t,x))^2 + q(t,x)\bigg]u(t,x) &= 0, \quad &(t,x)\in \Omega_T \\
	 u(0,x) &= 0,  \quad & x \in \Omega   \\
	 u(t,x) &= f(t,x), \quad & (t,x) \in \partial\Omega_T.
	\end{array}
	\right.
\end{align}
To discuss the well-posedness of the forward problem, which would help us to define our boundary operators,  we start with defining the following spaces which were initially used in \cite{caro2018determination}:
\begin{align*}
    \mathcal{K}_0 &:= \{f|_{\partial\Omega_T} : f \in L^2(0,T;H^1(\Omega)) \cap H^1(0,T:H^{-1}(\Omega)) \text{ and } f(0,x)=0, \text{ for } x\in \Omega\}.\\
    \mathcal{H}_T &:= \{g|_{\partial\Omega_T} : g \in H^1(0,T;H^1(\Omega))  \text{ and } g(T,x)=0, \text{ for }  x\in \Omega\}.
\end{align*}
As shown in \cite{caro2018determination} (see also \cite{lions1968problèmes,sahoo2019partial}) that the system \eqref{eq:1.1} admits a unique solution $u  \in L^2(0,T;H^1(\Omega)) \cap H^1(0,T;H^{-1}(\Omega))$ for $f\in\mathcal{K}_0$. Following the same references, we also define the operator $\mathcal{N}_{A,q}$ by
\begin{align}
\label{eq:1.2}
\langle \mathcal{N}_{A,q}u,v|_{\partial\Omega_T}\rangle := \int_{\Omega_T} (-u \partial_t \overline{v}  + \nabla u\cdot \nabla \overline{v} + 2uA\cdot \nabla \overline{v} + (\nabla\cdot A)u \overline{v} - |A|^2u\overline{v} + qu\overline{v} )dx dt
\end{align}
where $v \in H^1(\Omega_T)$ is such that $v(T,x) = 0$ for $x\in \Omega$.\\
\vspace{.5mm}
\\
The operator $\Nc_{A, q}$ takes the following simple form for smooth enough $A, q$, and $f$:
\begin{align*}
\mathcal{N}_{A,q}u = \left.(\partial_\nu u + 2(\nu\cdot A) u)\right|_{\partial\Omega_T}
\end{align*}
where $\nu$ denotes the outward unit normal vector to $\partial\Omega$. 
This motivates to define the Dirichlet to Neumann (DN) map $\Lambda_{A,q}:\mathcal{K}_0 \to \mathcal{H}^*_T$ by
\begin{align}
\label{eq:1.3}
    \Lambda_{A,q}(f) := \mathcal{N}_{A,q}u = \left.(\partial_\nu u + 2(\nu\cdot A) u)\right|_{\partial\Omega_T}.
\end{align}
where $\mathcal{H}^*_T$ denotes the dual of space $\mathcal{H}_T$ and $u$ is the solution to \eqref{eq:1.1} with Dirichlet boundary data equal to $f$. 

The present article focuses on the unique recovery of the coefficients $A$ and $q$ appearing in 
$\mathcal{L}_{A,q}$ from the knowledge of $\Lambda_{A,q}$ measured on a suitable  subset of the boundary. We can show that the full recovery of the convection term $A$ and $q$ is not possible since $\Lambda_{A+\nabla \Psi,q+\partial_t\Psi} = \Lambda_{A,q}$ when $\Psi \in W_0^{2,\infty}(\Omega_T)$. Therefore, the best we can hope is to recover $A$ and $q$ up to this natural gauge.

This problem of determining coefficients in parabolic partial differential equations by boundary measurements is an important problem and has been studied by many authors. Cannon, \cite{cannon1963determination} determined an unknown coefficient appearing in a one-dimensional parabolic partial differential equation. Isakov \cite{isakov1991completeness} studied the problem of determining a time-independent coefficient $q(x)$ for the case when $A = 0$ in \eqref{eq:1.1} from the DN map. Following this, Cheng and Yamamoto \cite{cheng2002identification} considered the problem of determining the convection term $A(x)$ when $q = 0$ in \eqref{eq:1.1} from a single boundary measurement in two dimensions. In \cite{deng2008identifying} identification of the first-order coefficient in parabolic equation from final measurement data is studied. Choulli and Kian in \cite{choulli2018logarithmic} proved a logarithmic stability estimate for determining the time-dependent zero-order coefficient in a parabolic equation from a partial DN map. 
In \cite{sahoo2019partial}, a unique recovery of both convection term $A(t,x)$ and density $q(t,x)$ appearing in \eqref{eq:1.1} from the boundary measurement by using an exponential weight with linear phase is proved. Bellassoued and Rassas \cite{bellassoued2020stability} uniquely determined both convection term $A(x)$ and density coefficient $q(x)$ by the DN map, and they showed for $n\geq3$ a log-type stability estimate. In \cite{bellassoued2021stably}, Bellassoued and Fraj proved the logarithmic stability estimates in the determination of the two time-dependent first-order convection term and the scalar potential appearing in the dynamical convection-diffusion equation. Senapati and Vashisth \cite{senapati2021stability} proved the stability estimate for determining the convection term $A(t,x)$ and the density coefficient $q(t,x)$ from the knowledge of DN map measured on a portion that is slightly bigger than half of the lateral boundary. Recently, Fan and Duan in \cite{fan2021determining} proved the unique recovery of time-dependent density coefficient $q(t,x)$ for the case when $A = 0$ in \eqref{eq:1.1}. The present work extends the work of \cite{fan2021determining} to the recovery of both convection term $A(t,x)$ and $q(t,x)$ from the knowledge of DN map measured on a very small subset of the boundary. 
 
 We refer to \cite{bellassoued2016inverse,choulli1991abstract,mchoulli1991abstract,choulli2009introduction,choulli2012stability,gaitan2013stability,hoffmann2005generic,isakov1993uniqueness,kamynin2011unique,kamynin2010two,mishra2022inverse,murayama1981gel, NakamuraSasayama+2013+217+232, prilepko1992inverse} and references therein for more works related to the partial data inverse problems for the parabolic equations. We also refer to \cite{ibtissem2017,bellassoued2008stability,cheng2000global,cheng2004determination,ferreira2007determining,kachalov2001inverse,kenig2007determining,kian2016recovery,kian2016stability,kian2017unique,oksanen2017,kian2014carleman,krishnan2020inverse,mishra2021determining,rakesh1988uniqueness,senapati2020stability,sylvester1987global} for related works on inverse problems for the elliptic, hyperbolic and dynamical Schr{\"o}dinger partial differential equations.

To state our problem of interest, we begin by giving some definitions and notations. 
Inspired by \cite{kenig2007determining}, we denote by $ch(\Omega)$  the convex hull of $\Omega$ and for  $x_0 \in \mathbb{R}^n \setminus \overline{ch(\Omega)}$,  we define the front and back sides of the boundary $\partial\Omega$ with respect to $x_0$ by
\begin{align*}
    F(x_0) := \{x \in \partial\Omega: (x-x_0)\cdot\nu(x) \leq 0 \},\\
    B(x_0) := \{x \in \partial\Omega: (x-x_0)\cdot\nu(x) > 0 \}.
\end{align*}
where $\nu(x)$ is the outer unit normal to the boundary point $x$.

\vspace{1mm}
\noindent With this preparation, we are ready to state the main result of this article.
\begin{theorem}\label{th:main theorem}
Let $\Omega \subset \mathbb{R}^n, n \geq 3$ be an open simply connected bounded set with a smooth boundary. Let $q_1, q_2 \in L^{\infty}(\Omega_T)$ be two density coefficients and $A^{(1)}, A^{(2)} \in W^{1,\infty}(\Omega_T)$ be two real vector fields. Suppose that there exists open neighborhood $\widetilde{F}$ of the front side $F(x_0)$ where $x_0, F(x_0)$ defined as above, such that the DN maps corresponding to the operators $\mathcal{L}_{A^{(1)},q_1}$ and $\mathcal{L}_{A^{(2)},q_2}$ coincide on the parts of the boundary near $x_0$ in a sense that 
\begin{align}
\label{eq:1.5}
    \Lambda_{A^{(1)},q_1}(f) =  \Lambda_{A^{(2)},q_2}(f), \quad \forall x \in \widetilde{F}, \quad\forall f \in \mathcal{K}_0,
\end{align}
then there exists a function $\Psi\in W_0^{2,\infty}(\Omega_T)$, such that
\begin{align*}
    A^{(2)}(t,x) = A^{(1)}(t,x) + \nabla_x\Psi(t,x) \quad\text{and} \quad q_2(t,x) = q_1(t,x) + \partial_t\Psi(t,x) \quad \text{in} \quad \Omega_T
\end{align*}
provided $A^{(1)}(t, x) = A^{(2)}(t, x)$ for $(t, x) \in \partial\Omega_T$.
\end{theorem}
Note that the recovery of the density, as well as the convection term, is up to the natural obstruction to the uniqueness. These obstructions can be removed (or modified) with some prior assumptions on the convection term, as we discuss below, in the form of two corollaries. 
\begin{corollary}\label{Cor: Corollary 2}
    Let $\Omega, x_0, F(x_0)$ defined as above theorem \ref{th:main theorem}. Let $q_1, q_2 \in L^{\infty}(\Omega_T)$ and $A^{(1)}, A^{(2)} $ be two real vector fields independent of time. Suppose that there exist open neighborhood $\widetilde{F}$ of the front side $F(x_0)$ such that 
\begin{align*}
 \Lambda_{A^{(1)},q_1}(f) =  \Lambda_{A^{(2)},q_2}(f), \quad \forall x \in \widetilde{F}, \quad\forall f \in \mathcal{K}_0,
\end{align*}
then if $A^{(1)}$ and $A^{(2)}$ are viewed as 1-forms ($\sum_{j=1}^{n} A_j dx_j$), we have
\begin{align*}
    dA^{(1)} = dA^{(2)} \quad\text{and} \quad q_1(t, x) = q_2(t, x) \quad \text{in} \quad \Omega_T
\end{align*}
provided $A^{(1)}(x) = A^{(2)}(x)$ for $x \in \partial\Omega$.
\end{corollary}

\begin{corollary}\label{Cor: Corollary 3}
   Let $(A^{(1)},q_1), (A^{(2)}, q_2)$ be two sets of time-dependent coefficients such that $q_i \in L^{\infty}(\Omega_T)$ and $A^{(i)} \in W^{1,\infty}(\Omega_T)$ for $i = 1,2$. If
 \begin{align}\label{eq: divergence condition}
     \nabla_x\cdot A^{(1)}(t,x) = \nabla_x\cdot A^{(2)}(t,x) \quad (t,x)\in \Omega_T 
 \end{align} 
 and
\begin{align*}
 \Lambda_{A^{(1)},q_1}(f) =  \Lambda_{A^{(2)},q_2}(f), \quad \forall x \in \widetilde{F}, \quad\forall f \in \mathcal{K}_0,
\end{align*}
then we have
\begin{align*}
    A^{(1)}(t, x) = A^{(2)}(t, x) \quad\text{and} \quad q_1(t, x) = q_2(t, x) \quad \text{in} \quad \Omega_T
\end{align*}
provided $A^{(1)}(t, x) = A^{(2)}(t, x)$ for $(t, x) \in \partial\Omega_T$.
\end{corollary}
    The result mentioned in the above theorem can be seen as the extension of the work of \cite{fan2021determining} where a problem of recovering $q(t,x)$ appearing in $\mathcal{L}_{A,q}$ when $A=0$ is studied. The recovery of both $A$ and $q$ has been studied by several other people as well. For example, in \cite{sahoo2019partial}, a uniqueness result with a smallness assumption on $A$ is proved for recovering both $A$ and $q$ from boundary data measured on an open subset of boundary which is slightly bigger than half of the $\partial\Omega$ which was further improved in \cite{senapati2021stability} to derive a stability estimate for recovering both $A$ and $q$.  In this work, we were able to prove the unique recovery of $A$ and $q$ without the smallness assumption on $A$ and in our case the boundary data is measured on a sufficiently small open subset of the the boundary. Our result and idea of the proof are mainly inspired by the work of \cite{choulli2018logarithmic,fan2021determining,ferreira2007determining,sahoo2019partial}.

A brief outline of this article is as follows. In section \ref{Carleman estimate}, we prove an interior and boundary Carleman estimates for the operator $\mathcal{L}_{A,q}$, which will be helpful in the construction of the geometric optics solutions for $\mathcal{L}_{A,q}$ and its formal $L^2-$adjoint $\mathcal{L}_{A,q}^{*}$. Using the interior Carlemann estimate derived in \ref{Carleman estimate}, in section \ref{Construction of solutions}, we construct the geometric optics solutions for $\mathcal{L}_{A,q}$ and $\mathcal{L}_{A,q}^{*}$ by solving an eikonal and the transport equation. 
Section \ref{Proof of main theorem} is devoted to proving the main result given in Theorem \ref{th:main theorem} of the present article and in Section \ref{Proof of Corollary 2 and 3}, we prove the Corollary \ref{Cor: Corollary 2} and Corollary \ref{Cor: Corollary 3}.


\section{Carleman estimates}\label{Carleman estimate}
This section is devoted to proving interior and boundary Carleman estimates for the operator $\mathcal{L}_{A,q}$. These estimates are required to construct the geometric optics solutions for the operator $\mathcal{L}_{A,q}$. We start by defining some notations which will be used throughout the article.
\begin{align*}
    \|u\|^2 &= \|u\|^2_{L^2(\Omega_T)} = \int_{0}^{T}\int_{\Omega} |u(t,x)|^2 dx dt\\
    (u,v) &= \int_{0}^{T}\int_{\Omega} u\overline{v} dx dt \\
    \|u\|^2_{L^2(\Omega)} &= \int_{\Omega} |u(t,x)|^2 dx.
\end{align*}

Let us rewrite the operator $\Lc_{A,q}$ as follows:
\begin{align*}
    \mathcal{L}_{A,q}u 
    &=\left(\partial_t - \Delta  - 2 A\cdot\nabla - (\nabla\cdot A) - |A|^2 + q\right)u\\
    &= \Lc_A u + \Tilde{q}u 
\end{align*}
where $$\mathcal{L}_{A} u:=(\partial_t - \Delta - 2 A\cdot\nabla)u \text{ and } \Tilde{q} := -\nabla\cdot A - |A|^2 + q.$$

\noindent To prove the estimates, we will be using semiclassical notations. For $h > 0$, define $P_0 = h^2(\partial_t - \Delta-2A\cdot\nabla)$ and  $\phi(t,x) = \varphi(x)\eta(t;h)$,  $\phi \in C^{\infty}((0,T)\times\Tilde{\Omega};\mathbb{R})$, with $\nabla_x\phi(t,x) \neq 0$,  $\forall (t,x) \in  (0,T)\times\Tilde{\Omega}$, where $\Tilde{\Omega} \subset \Rb^n$ be an open set containing $\O$. Now, consider the conjugated operator
\begin{align}
\label{eq:2.1a}
e^{\frac{\phi}{h}}\circ P_0 \circ e^{-\frac{\phi}{h}}u &=  h^2(u_t - \Delta u) + 2h \nabla_x\phi\cdot\nabla u - h(\phi_t - \Delta\phi)u - \nabla_x\phi\cdot\nabla_x\phi u \nonumber \\
&\qquad + h^2\bigg(2A\cdot\frac{\nabla_x\phi}{h}u - 2A\cdot\nabla u\bigg)\nonumber\\
& = (P + iQ)u + iR_tu + Su
\end{align}
where 
\begin{subequations}
\begin{align}
\label{eq:2.1b}
P & = -h^2\Delta - \nabla_x\phi\cdot\nabla_x\phi, \\
\label{eq:2.1c}
Q & =  2h \nabla_x\phi\cdot D - i h \Delta\phi, \quad D = (D_{x_1},\dots ,D_{x_n})  \mbox{ and } D_{x_j} = \frac{1}{i}\partial_{x_j}\\
\label{eq:2.1d}
R_t & = -i h^2 \partial_t + ih\phi_t,\\
\label{eq:2.1e}
S & = h^2\bigg(2A\cdot\frac{\nabla_x\phi}{h} - 2A\cdot\nabla \bigg).
\end{align}
\end{subequations}
Note that  $P, Q$ are self-adjoint operators with the principal Weyl symbols
\begin{align*}
    p =\xi^2 - \nabla_x\phi\cdot\nabla_x\phi, \quad q = 2\nabla_x\phi\cdot \xi.
\end{align*}
Since $\nabla_x\phi(t,x) \neq 0$, therefore $\phi$ is a limiting Carleman weight in the sense that
\begin{align*}
    \{p,q\}(x,\xi) = 0,  \mbox{ for all } (x, \xi) \mbox{ such that } p(x,\xi) = q(x,\xi) = 0,
\end{align*}
where $\{p,q\}$ is the Poisson bracket.\\

\noindent Our first step is to derive the Carleman estimate for the operator $P_0$ with a weight function $\phi(t,x)$ and then use the obtained estimate to derive the required Carleman estimate for $\Lc_{A,q}$. To do so, $\phi$ needs to satisfy additional terminal conditions, that is $\eta(T) =\eta_t(T) = 0$. Therefore, we assume, $\phi(t,x) = \varphi(x)\sin(h^{2/5}(T-t)^2)$, the power $2/5$ of $h$ is needed to estimate various terms in the upcoming analysis.

\subsection{Interior Carleman estimate}
To obtain an interior Carleman estimate, assume $u\in 
 C^{\infty}_0(\Omega_T)$, for $\Omega \Subset \Tilde{\Omega}$. Then, we use
 $\Tilde{u}=e^{\frac{\phi}{h}}u,  \Tilde{v}=e^{\frac{\phi}{h}}v$ and
 $P_0u = v$ in \eqref{eq:2.1a} to get
 \begin{align}
      ((P + i Q) + iR_t + S)\Tilde{u} = \Tilde{v}.
 \end{align}
This relation gives the following straightforward inequality:
 \begin{align}
 \label{eq:2.3}
     \|\Tilde{v}\|^2 &= \|((P + iQ) + iR_t + S)\Tilde{u}\|^2\nonumber\\& \geq{\frac{1}{2}\|((P + iQ) + iR_t)\Tilde{u}\|^2 - \|S\Tilde{u}\|^2}.
 \end{align}
We will estimate both the terms on the right-hand side of the above equation. Consider 
\begin{align}
 \|\left((P + i Q) + iR_t\right)\Tilde{u}\|^2&= \langle \left(P + i Q)+ iR_t\right)\Tilde{u}, ((P + i Q) + iR_t)\Tilde{u})\rangle \nonumber\\
 &= \langle (P + i Q)\Tilde{u},(P + i Q)\Tilde{u}\rangle + \langle(P + i Q)\Tilde{u},iR_t\Tilde{u}\rangle\nonumber \\&\quad\qquad+ \langle iR_t\Tilde{u},(P + i Q)\Tilde{u}\rangle + \langle R_t\Tilde{u},R_t\Tilde{u}\rangle\nonumber\\
 &= \|P\Tilde{u}\|^2 + \|Q\Tilde{u}\|^2 + i\langle [P,Q]\Tilde{u},\Tilde{u}\rangle  + \|R_t\Tilde{u}\|^2\nonumber\\ &\qquad\quad+ 2 Re\langle (P + i Q)\Tilde{u},iR_t\Tilde{u}\rangle   \label{eq:2.4}.
\end{align}
Since $u\in C^{\infty}_0((0,T)\times\Omega),\  \Omega \Subset \Tilde{\Omega}$, we use  \cite[Equation 3.11]{kenig2007determining} to get 
\begin{align}
\label{eq:2.5}
    \|P\Tilde{u}\|^2 + \|Q\Tilde{u}\|^2 + i\langle[P,Q]\Tilde{u},\Tilde{u}\rangle\geq \frac{2}{3}(\|P\Tilde{u}\|^2 + \|Q\Tilde{u}\|^2) - O(h^2)\|\Tilde{u}\|^2.
\end{align}
The next aim is to estimate the last two terms of \eqref{eq:2.4} by using the operators defined in \eqref{eq:2.1b}--\eqref{eq:2.1d}.
Consider 
\begin{align*}
2 Re\langle (P + i Q)\Tilde{u},iR_t\Tilde{u}\rangle  &= 2 Re\int_{0}^{T}\int_{\Omega} (P + i Q)\Tilde{u}\overline{iR_t\Tilde{u}} \ dx dt \\
&= 2 Re \int_{0}^{T}\int_{\Omega} -h^4\Delta\Tilde{u}\overline{\Tilde{u}_t} \ dxdt + 2Re\int_{0}^{T}\int_{\Omega}-h^2|\nabla_x\phi|^2\Tilde{u}\overline{\Tilde{u}_t} \ dxdt \\ & \quad + 2Re \int_{0}^{T}\int_{\Omega}2h^3 \nabla_x\phi\cdot\nabla\Tilde{u}\overline{\Tilde{u}_t} \ dxdt + 2 Re \int_{0}^{T}\int_{\Omega}h^3\Delta\phi\Tilde{u}\overline{\Tilde{u}_t} \ dxdt \\ & \quad +  2 Re \int_{0}^{T}\int_{\Omega}h^3\phi_t\Delta\Tilde{u}\overline{\Tilde{u}} \ dxdt + 2 Re \int_{0}^{T}\int_{\Omega}h|\nabla_x\phi|^2\phi_t\Tilde{u}\overline{\Tilde{u}} \ dxdt \\&\quad + 2 Re \int_{0}^{T}\int_{\Omega}-2h^2\phi_t\nabla_x\phi\cdot\nabla\Tilde{u}\overline{\Tilde{u}} \ dxdt + 2 Re \int_{0}^{T}\int_{\Omega}-h^2\phi_t\Delta\phi\Tilde{u}\overline{\Tilde{u}} \ dxdt\\
    &= I_1 + I_2 + \cdots + I_8.
\end{align*}
Now, we need to estimate the right-hand side term by term of the above relation.
\begin{align*}
I_1 &= 2 Re \int_{0}^{T}\int_{\Omega} -h^4\Delta\Tilde{u}\overline{\Tilde{u}_t} \ dx dt  = 2 Re \int_{0}^{T}\int_{\Omega} h^4 \nabla \Tilde{u}\cdot \nabla \overline{\Tilde{u}_t} \ dx dt\\
    &= \int_{0}^{T}\int_{\Omega} (h^4\nabla\Tilde{u}\cdot\overline{\Tilde{u}_t} + h^4 \nabla \Tilde{u}_t \cdot \nabla\overline{\Tilde{u}} ) \ dx dt = \int_{0}^{T}\int_{\Omega} h^4 |\nabla \Tilde{u} |^2_t \ dx dt = 0.\\
I_2 &= 2Re\int_{0}^{T}\int_{\Omega}-h^2|\nabla_x\phi|^2\Tilde{u}\overline{\Tilde{u}_t} \ dx dt = \int_{0}^{T}\int_{\Omega}(-h^2|\nabla_x\phi|^2\Tilde{u}\overline{\Tilde{u}_t} 
 - h^2|\nabla_x\phi|^2\overline{\Tilde{u}}\Tilde{u}_t) \ dxdt\\
 &= \int_{0}^{T}\int_{\Omega}-h^2|\nabla_x\phi|^2|\Tilde{u}|^2_t \ dx dt  = \int_{0}^{T}\int_{\Omega} h^2|\nabla_x\phi|^2_t|\Tilde{u}|^2\ dx dt \leq O(h^2) \|\Tilde{u}\|^2.\\
 I_3 &= 2Re \int_{0}^{T}\int_{\Omega}2h^3 \nabla_x\phi\cdot\nabla\Tilde{u}\overline{\Tilde{u}_t} \ dxdt \leq 4 h^3 \max_{(t,x)\in [0,T]\times\overline{\Omega}}|\nabla_x\phi| \int_{0}^{T}\int_{\Omega} |\nabla\Tilde{u}||\overline{\Tilde{u}_t}| \ dx dt\\
&\leq 4 h^3 \max_{(t,x)\in [0,T]\times\overline{\Omega}}|\nabla_x\phi| \||\nabla\Tilde{u}|\| \|\Tilde{u}_t\|.\\
I_4 &=   2 Re \int_{0}^{T}\int_{\Omega}h^3\Delta\phi\Tilde{u}\overline{\Tilde{u}_t} \ dxdt = \int_{0}^{T}\int_{\Omega}(h^3\Delta\phi\Tilde{u}\overline{\Tilde{u}_t} + h^3\Delta\phi\Tilde{u}_t\overline{\Tilde{u}}) \ dx dt\\
  &=  \int_{0}^{T}\int_{\Omega}h^3\Delta\phi|\Tilde{u}|^2_t \ dx dt \leq \max_{(t,x)\in [0,T]\times\overline{\Omega}} |(\Delta\phi)_t| \|\Tilde{u}\|^2\leq O(h^3)\|\Tilde{u}\|^2.\\
   I_5 &= 2 Re \int_{0}^{T}\int_{\Omega}h^3\phi_t\Delta\Tilde{u}\overline{\Tilde{u}} \ dxdt  = - 2 Re \int_{0}^{T}\int_{\Omega}h^3\phi_t\nabla\Tilde{u}\cdot\overline{\nabla\Tilde{u}} \ dxdt - 2 Re \int_{0}^{T}\int_{\Omega}h^3\nabla_x\phi_t\cdot \nabla\Tilde{u}\overline{\Tilde{u}} \ dxdt \hspace{1.1cm}\\
  & =- 2 Re \int_{0}^{T}\int_{\Omega}h^3\phi_t|\nabla\Tilde{u}|^2\ dx dt -2 Re \int_{0}^{T}\int_{\Omega}h^3\nabla_x\phi_t\cdot\nabla\Tilde{u}\overline{\Tilde{u}} \ dxdt \leq O(h^3)\||\nabla\Tilde{u}|\|^2 + O(h^2)\|\Tilde{u}\|^2.\\
  I_6 &= 2 Re \int_{0}^{T}\int_{\Omega}h|\nabla_x\phi|^2\phi_t\Tilde{u}\overline{\Tilde{u}} \ dxdt =  2 \int_{0}^{T}\int_{\Omega}h|\nabla_x\phi|^2\phi_t|\Tilde{u}|^2\ dxdt\\
& \leq  2 h\int_{0}^{T}\int_{\Omega} h^{4/5}T^4|\nabla\varphi|^2 h^{2/5}T^2\varphi|\Tilde{u}|^2 dx dt \leq O(h^2) \|\Tilde{u}\|^2.\\
 \end{align*}

\begin{align*}
I_7 &=  2 Re
\int_{0}^{T}\int_{\Omega}-2h^2\phi_t\nabla_x\phi\cdot\nabla\Tilde{u}\overline{\Tilde{u}} \ dxdt \leq 4h^2 \max_{(t,x)\in [0,T]\times\overline{\Omega}}|\phi_t|\max_{(t,x)\in [0,T]\times\overline{\Omega}}|\nabla_x\phi|\| |\nabla\Tilde{u}| \| \|\Tilde{u}\|\\
&\leq O(h^2)\| |\nabla\Tilde{u}| \|^2 + O(h^2) \|\Tilde{u}\|^2.\\
I_8 &= 2 Re \int_{0}^{T}\int_{\Omega} -h^2\phi_t\Delta\phi\Tilde{u}\overline{\Tilde{u}} \ dxdt \leq O(h^2)\|\Tilde{u}\|^2. 
\end{align*}
Hence by combining and using all the above estimates, we obtain
\begin{align}
\label{eq:2.6}
     2 Re\langle(A + i B)\Tilde{u},iR_t\Tilde{u}\rangle \leq  O(h^2)\|\Tilde{u}\|^2 + 4h^3 \underset{(t,x)\in[0,T]\times\Omega}{\max}|\nabla_x\phi|\||\nabla\Tilde{u}|\|\|\Tilde{u}_t\|+O(h^2)\||\nabla\Tilde{u}|\|^2.
\end{align}
Next, we use the definition of the operator \eqref{eq:2.1d} to obtain
\begin{align}
\label{eq:2.7}
 \|R_t\Tilde{u}\|^2 = \int_{0}^{T}\int_{\Omega}R_t\Tilde{u}\overline{R_t\Tilde{u}} \ dxdt \geq h^4\|\Tilde{u}_t\|^2 - O(h^2)\|\Tilde{u}\|.
\end{align}
Substituting the obtained estimates in \eqref{eq:2.4} from \eqref{eq:2.5} and then combining with \eqref{eq:2.6} and \eqref{eq:2.7}, we get
\begin{align*}
     \|((P + iQ) + iR_t)\Tilde{u}\|^2 &\geq  \frac{2}{3}(\|P\Tilde{u}\|^2 + \|Q\Tilde{u}\|^2) +  h^4\|\Tilde{u}_t\|^2 - 4h^3 \underset{(t,x)\in[0,T]\times\Omega}{\max}|\nabla_x\phi|\||\nabla\Tilde{u}|\|\|\Tilde{u}_t\| \\
     & \qquad -O(h^2)\||\nabla\Tilde{u}|\|^2 - O(h^2)\|\Tilde{u}\|.
\end{align*}
Using the Cauchy-Schwarz inequality,
\begin{align}
\label{eq:Cauchy_Schwarz_inequality}
4h^3 \underset{(t,x)\in[0,T]\times\Omega}{\max}|\nabla_x\phi|\||\nabla\Tilde{u}|\|\|\Tilde{u}_t\| \leq \frac{1}{2}h^4\|\Tilde{u}_t\|^2 + O(h^2)\||\nabla\Tilde{u}|\|^2.
\end{align}
and a priori estimate (followed by ellipticity of $P$) for a fixed $t$,
\begin{align}
\label{eq:2.8}
\|h|\nabla\Tilde{u}|\|^2_{L^2{(\Omega)}}\leq O(1)\big(\|P\Tilde{u}\|^2_{L^2{(\Omega)}} + \|\Tilde{u}\|^2_{L^2{(\Omega)}}\big).
\end{align}
Finally, we have the following estimate for the first term on the right side of \eqref{eq:2.3}:
\begin{align}
\label{eq:2.9}
     \|((P + i Q) + iR_t)\Tilde{u}\|^2\geq \frac{1}{2}(\|P\Tilde{u}\|^2 + \|Q\Tilde{u}\|^2) + \frac{1}{2}h^4\|\Tilde{u}_t\|^2 - O(h^2)\|\Tilde{u}\|^2.
\end{align}
Now, we estimate the second term on the right side of \eqref{eq:2.3} with the help of the operator \eqref{eq:2.1e}
\begin{align}
    \|S\Tilde{u}\|^2 & = \int_{0}^{T}\int_{\Omega}|2hA\cdot\nabla_x\phi\Tilde{u}-2h^2A\cdot\nabla\Tilde{u}|^2\ dx dt  \leq 8 h^2 \|A\|^2_{\infty}(\|h\nabla\Tilde{u}\|^2  + \underset{(t,x)\in[0,T]\times\Omega}{\max}|\nabla_x\phi|\|\Tilde{u}\|^2) \nonumber\\
    & \leq O(h^2)(\|P\Tilde{u}\|^2 + \|\Tilde{u}\|^2). \quad (\text{by using the relation \eqref{eq:2.8}})    \label{eq:2.10}
\end{align}
By using estimates \eqref{eq:2.9} and \eqref{eq:2.10} in Equation \eqref{eq:2.3}, we obtain
\begin{align}
\label{eq:2.11}
    \|\Tilde{v}\|^2\geq C(\|P\Tilde{u}\|^2 + \|Q\Tilde{u}\|^2) + \frac{1}{4}h^4\|\Tilde{u}_t\|^2 - O(h^2)\|\Tilde{u}\|^2,
\end{align}
for some constant $C > 0$.\\
\vspace{.2mm}\\
We need to modify the Carleman weight $\phi$ to absorb $O(h^2)\|\Tilde{u}\|^2$ in the above relation \eqref{eq:2.11}. To do this, let us consider a modified weight
\begin{align}
\label{eq:2.12}
    \phi_{\epsilon} = \phi + \frac{\epsilon \phi^2}{2}, \qquad 0\leq \epsilon \ll 1.
\end{align}
Corresponding to the modified weight $\phi_{\epsilon}$, the operators $P, Q, R_t$ and $S$ become $P_{\epsilon}, Q_{\epsilon},R_{t\epsilon}$ and $S_{\epsilon}$ respectively. For this modified Carleman weight, let $\hat{u} = e^{\phi_{\epsilon}/h}u$, $\hat{v} = e^{\phi_{\epsilon}/h}v$, when $P_0u = v$. Following similar calculations, we can 
\begin{align}
\label{eq:2.13}
\|\hat{v}\|^2 & \geq \frac{h}{2}(4\epsilon + O(\epsilon^2))\int_{0}^{T}\int_{\Omega}|\nabla_x\phi|^4|\hat{u}|^2 dxdt + 
\frac{1}{4}(\|P_{\epsilon}\hat{u}\|^2 + \|Q_{\epsilon}\hat{u}\|^2) + \frac{1}{4}h^4\|\hat{u}_t\|^2 - O(h^2)\|\hat{u}\|^2\nonumber \\
& - 8 h^2 \|P_{\epsilon}\|^2_{\infty}\|h\nabla\hat{u}\|^2  - 8 h^2 \|P_{\epsilon}\|^2_{\infty}\underset{(t,x)\in[0,T]\times\Omega}{\max}|\nabla_x\phi_\epsilon|\|\hat{u}\|^2. 
\end{align}
If we choose  $h \ll \epsilon^2 \ll 1$ then the relation \eqref{eq:2.13} gives
\begin{align}
\label{eq:2.14}
 \|\hat{v}\|^2 & \geq \frac{h\epsilon}{C_0}(\|\hat{u}\|^2 + \|hD\hat{u}\|^2) + \bigg(\frac{1}{4} - O(\epsilon h)\bigg)\|P_{\epsilon}\hat{u}\|^2 + \frac{1}{4}\|Q_{\epsilon}\hat{u}\|^2 + \frac{1}{4}h^4\|\hat{u}_t\|^2. 
\end{align}
To get the above inequality in terms of $\Tilde{u}$ and $\Tilde{v}$, we write $\phi_{\epsilon} = \phi + \epsilon \phi_1$, where $\phi_1$ and all its derivatives are bounded. With this change, we have $\hat{u} = e^{\frac{\epsilon \phi_1}{h}}\Tilde{u}, \hat{v} = e^{\frac{\epsilon \phi_1}{h}}\Tilde{v}$ and the following estimate (for details, please see \cite[Page 9]{fan2021determining}): 

\begin{align}
   \|e^{\frac{\epsilon \phi_1}{h}}\Tilde{v}\|^2 \geq \frac{h\epsilon}{C_0}\bigg(\|e^{\frac{\epsilon \phi_1}{h}}\Tilde{u}\|^2 + \|e^{\frac{\epsilon \phi_1}{h}}hD\Tilde{u}\|^2\bigg). 
\end{align}
Let us take $\epsilon = C'h$ with fixed $C'\gg1$, we get $e^{\frac{\epsilon \phi_1}{h}}$ is uniformly bounded in $\Omega_T$ and therefore we obtain
\begin{align}
\label{eq:2.16}
    h^2\left(\|\Tilde{u}\|^2 + \|hD\Tilde{u}\|^2\right) \leq C\|\Tilde{v}\|^2.
\end{align}
This gives the interior Carleman estimate for the operator $P_0$. The required interior Carleman estimate for the operator $P_0 + h^2 \Tilde{q}$ ($\Tilde{q} \in L^\infty(\O_T)$) is obtained by replacing $\Tilde{v}-h^2\Tilde{q}\Tilde{u}$ for $\Tilde{v}$ in \eqref{eq:2.16}. More specifically, the required interior Carleman estimate is given by
\begin{align}
\label{eq:2.18}
   h\left(\|e^{\frac{\phi}{h}}u\| + \|hD(e^{\frac{\phi}{h}}u)\|\right) \leq C\|e^{\frac{\phi}{h}} h^2 \mathcal{L}_{A,q}u\| 
\end{align}
where $C = C(\Omega)$ and $h>0$ is small enough.\\
We can also derive the interior Carleman estimate for $e^{-\frac{\phi}{h}}u$, by following a similar calculation as above, we get (same as changing $\phi$ into $-\phi$ in \eqref{eq:2.18})
\begin{align}
\label{eq:2.19}
 h\left(\|e^{-\frac{\phi}{h}}u\| + \|hD(e^{-\frac{\phi}{h}}u)\|\right) \leq C\|e^{-\frac{\phi}{h}} h^2 \mathcal{L}_{A,q}u\| 
\end{align}
for some constant $C > 0$ independent of $h$.

\subsection{Boundary Carleman estimate}
In this subsection, we focus on Carleman estimates involving the boundary terms, which will be needed to estimate the boundary terms appearing in the integral identity where we do not have the measurements. Let $P_0u=v$, for $u \in C^{\infty}(\Omega_T)$ such that  $u|_{\partial\Omega} = 0$ and $ u(0,x) = 0$. Further, we take
\begin{align*}
    \hat{u} = e^{\frac{\phi}{h}}u, \quad \hat{v} = e^{\frac{\phi}{h}}v, \quad \phi = \phi_{\epsilon}, 
\end{align*}
with
\begin{align*}
    P = P_{\epsilon},\quad Q = Q_{\epsilon},\quad R_t = R_{t\epsilon}, \quad S = S_{\epsilon}
\end{align*}
Then, $\hat{u}$ and $\hat{v}$ satisfy the following equation:
\begin{align}
   \left((P + i Q) + iR_t + S\right)\hat{u} = \hat{v},
\end{align}
which implies
\begin{align}
\label{eq:2.20}
    \|\hat{v}\|^2 = \|((P + i Q) + iR_t + S)\hat{u}\|^2 \geq{\frac{1}{2}\|((P + i Q) + iR_t)\hat{u}\|^2 - \|S\hat{u}\|^2}.
\end{align}
Now consider
\begin{align}
\|((P + i Q) + iR_t)\hat{u}\|^2 & = \langle((P + i Q) + iR_t)\hat{u}, ((P + i Q) + iR_t)\hat{u}\rangle \nonumber\\
 & = \langle(P + i Q)\hat{u},(P + i Q)\hat{u}\rangle + \langle(P + i Q)\hat{u},iR_t\hat{u}\rangle + \langle iR_t\hat{u},(P + i Q)\hat{u}\rangle\nonumber\\&\qquad +\|R_t\hat{u}\|^2\nonumber\\
 \label{eq:2.21}
 & = \|P\hat{u}\|^2 + \|Q\hat{u}\|^2 + i (\langle Q\hat{u},P\hat{u}\rangle -\langle P\hat{u},Q\hat{u}\rangle)  + 2 Re\langle (P + i Q)\hat{u},iR_t\hat{u}\rangle \nonumber\\
 &\qquad + \|R_t\hat{u}\|^2.
\end{align}
Using the definition of operator $B$ and $\hat{u}|_{\partial\Omega} = e^{\frac{\phi}{h}}u|_{\partial\Omega} = 0,$ we get
\begin{align}
\label{eq: IP(A,B)}
 \langle P\hat{u},Q\hat{u}\rangle& = \langle QP\hat{u},\hat{u}\rangle,\\
 \langle Q\hat{u},P\hat{u}\rangle& = \langle Q\hat{u},-h^2\Delta\hat{u}\rangle+ \langle Q\hat{u},(\nabla_x\phi)^2\hat{u}\rangle.\nonumber
\end{align}
Applying integration by parts to get
\begin{align*}
    \langle Q\hat{u},-h^2\Delta\hat{u}\rangle = \langle -h^2
    \Delta Q\hat{u},\hat{u}\rangle + \langle -h^2Q\hat{u},\partial_{\nu}\hat{u}\rangle_{\partial\Omega_T}.
\end{align*}
We decompose $Q$ into its tangential and normal components as follows:
\begin{align*}
    Q = 2(\nabla_x\phi\cdot\nu)\frac{h}{i}\partial_\nu+Q^\prime 
\end{align*}
With this decomposition and boundary condition $\hat{u}|_{\partial\Omega} = 0$, we get
\begin{align*}
\langle -h^2Q\hat{u},\partial_{\nu}\hat{u}\rangle_{\partial\Omega_T} = -\frac{2h^3}{i}\langle(\nabla_x\phi\cdot\nu)\partial_{\nu}\hat{u},\partial_{\nu}\hat{u}\rangle_{\partial\Omega_T}.
\end{align*}
Combining the above equalities to obtain
\begin{align}
\label{eq:2.22}
    \langle Q\hat{u},P\hat{u}\rangle=\langle PQ\hat{u},\hat{u}\rangle -\frac{2h^3}{i}\langle (\nabla_x\phi\cdot\nu)\partial_{\nu}\hat{u},\partial_{\nu}\hat{u}\rangle_{\partial\Omega_T}.
\end{align}
With the help of \eqref{eq: IP(A,B)} and \eqref{eq:2.22}, we can rewrite Equation \eqref{eq:2.21} as follows:
\begin{align}
\label{eq:2.23}
\|((P + iQ) + iR_t)\hat{u}\|^2  = &  \|P\hat{u}\|^2 + \|Q\hat{u}\|^2 +i\langle [P,Q]\hat{u},\hat{u}\rangle-2h^3\langle (\nabla_x\phi\cdot\nu)\partial_{\nu}\hat{u},\partial_{\nu}\hat{u}\rangle_{\partial\Omega_T}+\|R_t\hat{u}\|^2 \nonumber\\ & + 2 Re\langle (P + iQ)\hat{u},iR_t\hat{u}\rangle.
\end{align}
Recall $\phi(t, x) =  \varphi(x) \eta(t;h)$, with $\eta(T) = \eta_t(T) = 0$. Therefore, we have $\phi(T,x) = \nabla_x\phi(T,x) = \Delta\phi(T,x) = \phi_t(T,x) = 0.$
With these conditions in mind and following a similar calculation as done previously (to obtain \eqref{eq:2.6}),  we have
\begin{align}
\label{eq:2.24}
2 Re\langle(P + i Q)\hat{u},iR_t\hat{u}\rangle &\leq h^4\int_{\Omega}|\nabla\hat{u}(T,x)|^2 dx + O(h^2)\||\nabla\hat{u}|\|^2 + O(h^2)\|\hat{u}\|^2 \nonumber \\
&\quad + 4h^3 \underset{(t,x)\in[0,T]\times\Omega}{\max}|\nabla_x\phi|\||\nabla\hat{u}|\|\|\hat{u}_t\|.
\end{align}
Using the definition of the operator $R_t$, we have
\begin{align}
\label{eq:2.25}
 \|R_t\hat{u}\|^2 \geq h^4\|\hat{u}_t\|^2-O(h^2)\|\hat{u}\|^2.
\end{align}
From \eqref{eq:2.24} and \eqref{eq:2.25} with using the Cauchy-Schwartz inequality \eqref{eq:Cauchy_Schwarz_inequality} for $\hat{u}$, we get that
\begin{align}
\label{eq:2.26}
    \|R_t\hat{u}\|^2 + 2 Re\langle(P + i Q)\hat{u},iR_t\hat{u}\rangle \geq \frac{1}{2}h^4\|\hat{u}_t\|^2 + h^4\int_{\Omega}|\nabla\hat{u}(T,x)|^2 dx - O(h^2)\||\nabla\hat{u}|\|^2 - O(h^2)\|\hat{u}\|^2.
\end{align}
Now consider
\begin{align}
\label{eq:2.27}
\|S\hat{u}\|^2 & = \int_{0}^{T}\int_{\Omega}|2hA\cdot\nabla_x\phi\hat{u}-2h^2A\cdot\nabla\hat{u}|^2 dx dt \leq 8 h^2 \|A\|^2_{\infty}(\|h\nabla\hat{u}\|^2  + \underset{(t,x)\in[0,T]\times\Omega}{\max}|\nabla_x\phi|\|\Tilde{u}\|^2) \nonumber\\
    & \leq O(h^2)(\|P\hat{u}\|^2 + \|\hat{u}\|^2).
\end{align}
Let
\begin{align}
\label{eq:2.28}
    \partial\Omega_{\pm} := \{x\in\partial\Omega: \pm\nabla\varphi\cdot\nu\geq 0\}.
\end{align}
Substituting \eqref{eq:2.23} combined with \eqref{eq:2.26}--\eqref{eq:2.28} in \eqref{eq:2.20}, we get
\begin{align}
     \|\hat{v}\|^2 &+ h^3\langle(\nabla_x\phi\cdot\nu)\partial_{\nu}\hat{u},\partial_{\nu}\hat{u}\rangle_{(0,T)\times\partial\Omega_+} \nonumber \\& \geq C\|P\hat{u}\|^2 + \frac{1}{4}\|Q\hat{u}\|^2 + \frac{1}{4}h^4\|\hat{u}_t\|^2 - O(h^2)\|\hat{u}\|^2 - h^3\langle(\nabla_x\phi\cdot\nu)\partial_{\nu}\hat{u},\partial_{\nu}\hat{u}\rangle_{(0,T)\times\partial\Omega_-}\nonumber\\
     & \geq \frac{h\epsilon}{C_0}(\|\hat{u}\|^2 + \|hD\hat{u}\|^2) + \big(C - O(\epsilon h)\big)\|P_{\epsilon}\hat{u}\|^2 + \frac{1}{4}\|Q_{\epsilon}\hat{u}\|^2 - O(h^2)\|\hat{u}\|^2\nonumber \\
     &\qquad - h^3\langle(\nabla_x\phi\cdot\nu)\partial_{\nu}\hat{u},\partial_{\nu}\hat{u}\rangle_{(0,T)\times\partial\Omega_-}.
\end{align}
Let us fix $\epsilon = C'h$ for some $C'\gg1$, with $\phi = \phi_{\epsilon_0}$ for some $C_0 > 0$, we get
\begin{align}
\label{eq:2.30}
    -\frac{h^3}{C_0}\langle\left(\nabla_x\phi\cdot\nu\right)\partial_{\nu}\Tilde{u},\partial_{\nu}\Tilde{u}\rangle_{(0,T)\times\partial\Omega_-} &+ \frac{h^2}{C_0}\left(\|\Tilde{u}\|^2 + \|hD\Tilde{u}\|^2\right) \nonumber \\ &\qquad \leq \|\Tilde{v}\|^2 + C_0h^3\langle(\nabla_x\phi\cdot\nu)\partial_{\nu}\Tilde{u},\partial_{\nu}\Tilde{u}\rangle_{(0,T)\times\partial\Omega_+}.
\end{align}
Note that this is a boundary Carleman estimate for the operator $P_0$. The required boundary Carleman estimate for the operator $P_0 + h^2\Tilde{q}$ for $\Tilde{q} \in L^{\infty}(\Omega_T)$ is obtained by replacing $\Tilde{v}$ by $\Tilde{v}-h^2\Tilde{q}\Tilde{u}$ in  \eqref{eq:2.30}. Therefore, the required boundary Carleman estimate is given by
\begin{align}
\label{eq:2.31}
    -\frac{h^3}{C_0}\langle(\nabla_x\phi\cdot\nu)e^{\phi/h}\partial_{\nu}u,e^{\phi/h}\partial_{\nu}u\rangle_{(0,T)\times\partial\Omega_-} + \frac{h^2}{C_0}(\|e^{\phi/h}u\|^2 + \|e^{\phi/h}hDu\|^2)\nonumber \\ \leq \|e^{\phi/h}h^2(\mathcal{L}_{A,q}u)\|^2 + C_0h^3\langle(\nabla_x\phi\cdot\nu)e^{\phi/h}\partial_{\nu}u,e^{\phi/h}\partial_{\nu}u\rangle_{(0,T)\times\partial\Omega_+}
\end{align}
where $C_0 > 0$ and $0 < h \ll 1$.
This gives a Carleman estimate for $e^{\phi/h}u,$ when $h^2\mathcal{L}_{A,q}u = v$ with $u|_{\partial\Omega} = 0$ and $u(0) = 0.$

We can derive another Carleman estimate by following exactly similar calculations for $e^{-\phi/h}u$ (changing $\phi$ into $-\phi$ in \eqref{eq:2.31}),
\begin{align}
 \label{eq:2.32}
    \frac{h^3}{C_0}\langle(\nabla_x\phi\cdot\nu)e^{-\phi/h}\partial_{\nu}u,e^{-\phi/h}\partial_{\nu}u\rangle_{(0,T)\times\partial\Omega_+} + \frac{h^2}{C_0}(\|e^{-\phi/h}u\|^2 + \|e^{-\phi/h}hDu\|^2)\nonumber \\ \leq \|e^{-\phi/h}h^2(\mathcal{L}_{A,q}u)\|^2 - C_0h^3\langle(\nabla_x\phi\cdot\nu)e^{-\phi/h}\partial_{\nu}u,e^{-\phi/h}\partial_{\nu}u\rangle_{(0,T)\times\partial\Omega_-}.
\end{align}


\section{Construction of  geometric optics solutions}\label{Construction of solutions}
Let us denote the semiclassical Sobolev space of order $s$ on $\mathbb{R}^n$ by $H^s_{scl}(\mathbb{R}^n)$ with the norm
\begin{align*}
    \|u\|^2_{H^s_{scl}(\mathbb{R}^n)} = \|\langle hD\rangle^s u\|^2_{L^2(\mathbb{R}^n)} = \int_{\Rb^n} \big(1 + h^2\xi^2\big)^s |\hat{u}(\xi)|^2\ d\xi,
\end{align*}
And similarly, $H^1_{scl}(\Omega)$ is the Sobolev space on $\Omega$ of order $1$ with the corresponding norm given by
\begin{align*}
    \|u\|^2_{H^1_{scl}(\Omega)} =  \|h\nabla u\|^2_{L^2(\Omega)} + \|u\|^2_{L^2(\Omega)}.
\end{align*}
Further for $s \in \mathbb{R}$, we define the space $L^2(0,T;H^s_{scl}(\mathbb{R}^n))$ by
\begin{align*}
L^2(0,T;H^s_{scl}(\mathbb{R}^n)):= \big\{u(t,\cdot)\in S'(\mathbb{R}^n):\big(1 +|h\xi|^2\big)^{s/2}\hat{u}(t,\xi) \in L^2(\mathbb{R}^n)\big\},
\end{align*}
with the norm 
\begin{align*}
\|u\|^2_{L^2(0,T;H^s_{scl}(\mathbb{R}^n))}:=\int_{0}^{T}\int_{\mathbb{R}^n} \big(1 +|h\xi|^2\big)^s|\hat{u}(t,\xi)|^2\ d\xi dt,
\end{align*}
 where $S'(\mathbb{R}^n)$ denote the space of all tempered distribution on $\mathbb{R}^n$ and $\hat{u}(t,\xi)$ is the Fourier transform with respect to space variable $x\in \mathbb{R}^n$.
 \\
\par The goal of this section is to construct solutions for the convection-diffusion equation. To do so, we take $\psi$ to be the solution of the eikonal equation:
\begin{align}
    |\nabla\psi|^2 = |\nabla\varphi|^2, \quad \nabla\varphi\cdot\nabla\psi = 0.
\end{align}
We will be working with a fixed limiting Carleman weight given by
\begin{align}
\label{eq:3.2}
    \varphi(x) = \frac{1}{2}\log(x-x_0)^2.
\end{align}
Given the above choice of $\varphi$, we have an explicit solution of the eikonal equation
\begin{align}
\label{eq:3.3}
    \psi(x) = d_{\Sb^{n-1}}\bigg(\frac{x-x_0}{|x-x_0|},\omega\bigg), \quad \mbox{ where } \o\in \Sb^{n-1}.
\end{align}
Following the ideas developed in \cite{ferreira2007determining}, for $r_0 > 0$ large enough so that $\overline{\Omega} \subset B(x_0,r_0)$, we define $\Gamma = \{\theta \in \Sb^{n-1} : x_0+r_0\theta \in H^+\}$,  where $H$ is a hyperplane separating $x_0$ and $ch(\Omega)$, and $H^+$ is the open half-space, containing $\overline{\Omega}$ (therefore, $x_0 \notin H^+$). Further, let $\check{\Gamma}$ be the image of $\Gamma$ under the antipodal map and $\Gamma_0$ be a neighborhood of $\omega_0$ in $ \Sb^{n-1}\setminus(\Gamma \cup \check{\Gamma})$. Then the distance map defined on $\Gamma\times \Gamma_0$ by  $(\theta, \omega) \mapsto d_{\Sb^{n-1}}(\theta, \omega)$ is a smooth function. Therefore, for $x \in x_0 + \mathbb{R}_+\Gamma$ we have $\frac{x-x_0}{|x-x_0|} \in \Gamma$ and hence $\psi$ depends smoothly on the variables $(x,\omega)$ on $\left(x_0 + \mathbb{R}_+\Gamma\right)\times \Gamma_0$.

\subsection{Construction of exponentially growing solution}
Recall, 
\begin{align}
\label{eq:3.4}
    \eta(t;h) = \sin{(h^{2/5}(T-t)^2)}.
\end{align}
For such $\varphi, \psi$ and $\eta$ defined above, we obtain
\begin{align*}
    h^2 \mathcal{L}_{A,q}e^{\frac{1}{h}(\varphi+i\psi)\eta}a = e^{\frac{1}{h}(\varphi+i\psi)\eta}(h^2 \mathcal{L}_{A,q} + h(\varphi+i\psi)\eta_t - h\eta L)a
\end{align*}
where $L$ is the transport operator given by
\begin{align*}
    L = (\Delta\varphi + i\Delta\psi) + 2(\nabla\varphi + i\nabla\psi)\cdot\nabla + 2A\cdot(\nabla\varphi+i\nabla\psi).
\end{align*}
We wish to find $a$ satisfying $La = 0$ of the form $a = m(t)e^{\Phi(t,x)}$ for any non-vanishing real-valued differentiable function $m(t)$.  In order to find such a function $a$, $\Phi(t,x)$ must satisfy 
\begin{align}
\label{eq:3.5}
     (\Delta\varphi + i\Delta\psi) + 2(\nabla\varphi + i\nabla\psi)\cdot\nabla\Phi + 2A\cdot(\nabla\varphi+i\nabla\psi) = 0.
\end{align}
For $\varphi+i\psi$ and $m(t)e^{\Phi(t,x)}$, we obtain an approximate solution of the convection-diffusion equation 
\begin{align}
\label{eq:3.6}
     h^2 \mathcal{L}_{A,q}\big(e^{\frac{1}{h}(\varphi+i\psi)\eta}m(t)e^{\Phi(t,x)}\big) = e^{\frac{1}{h}(\varphi+i\psi)\eta}(h^2 \mathcal{L}_{A,q} + h(\varphi+i\psi)\eta_t)m(t)e^{\Phi(t,x)} = e^{\frac{\varphi}{h}\eta(t;h)}O(h^{7/5}).
\end{align}
We can transform this approximate solution into an exact one using the following proposition.
\begin{Prop}\label{Proposition 1}
Let $\phi$ be a limiting Carleman weight and $q\in L^{\infty}(\Omega_T)$. Then for small enough $h>0$ and for all $F\in L^2(\Omega_T)$ there exists a solution $u\in L^2(0,T;H^1(\Omega))\cap H^{1}(0,T;H^{-1}(\Omega))$ of 
\begin{align*}
\left\{
	\begin{array}{ l l l }
  e^{-\frac{\phi}{h}}h^2\mathcal{L}_{A,q}  e^{\frac{\phi}{h}}u = F, \ (t,x)\in\Omega_T\\
    u(0,x) = 0, \ x\in \Omega 
\end{array}
\right.
\end{align*}
satisfying the estimate
\begin{align*}
    h\|u\|_{L^2(0,T;H^1_{scl}(\Omega))}\leq C\|F\|_{L^2(\Omega_T)}.
\end{align*}
\end{Prop}
\renewcommand\qedsymbol{$\square$}
\begin{proof}
To prove the existence, we first prove the following Carleman estimate
\begin{align*}
    \|v\|_{L^2(0,T;L^2(\mathbb{R}^n))} \leq C h \| e^{\frac{\phi}{h}}\mathcal{L}^{*}_{A,q} e^{-\frac{\phi}{h}}v\|_{L^2\big(0,T;H^{-1}_{scl}(\mathbb{R}^n)\big)}
\end{align*}
holds for $v \in C^1([0,T];C^{\infty}_0(\Omega))$ satisfying $v(T,x)=0$.\\
The interior Carleman estimate for the adjoint operator:
\begin{align}\label{adjoint interior carleman estimate}
    \|v\|_{L^2(0,T;H^{1}_{scl}(\mathbb{R}^n))}\leq C h \|e^{\frac{\phi}{h}} \mathcal{L}^{*}_{A,q}e^{-\frac{\phi}{h}}v\|_{L^2{(\Omega_T)}}
\end{align}
Let us consider a modified weight
\begin{align*}
    \phi_{\epsilon} = \phi + \frac{\epsilon \phi^2}{2}, \qquad 0\leq \epsilon \ll 1.
\end{align*}
By using modified weight, consider the conjugated adjoint operator
\begin{align*}
    e^{\frac{\phi_{\epsilon}}{h}}h^2\mathcal{L}^{*}_{A,q} e^{-\frac{\phi_{\epsilon}}{h}}v = &  h^2(-v_t - \Delta v) + 2h \nabla_x\phi_{\epsilon}\cdot\nabla v + h(\partial_t \phi_{\epsilon} + \Delta\phi_{\epsilon})v - \nabla_x\phi_{\epsilon}\cdot\nabla_x\phi_{\epsilon} v \\
    & + h^2\bigg(2A\cdot\nabla - 2A\cdot\frac{\nabla_x\phi_{\epsilon}}{h} \bigg)\nonumber v + \Tilde{q}^{*}v
\end{align*}
where $\Tilde{q}^{*} = \overline{q} + \nabla\cdot A -|A|^2$.\\
We rewrite the above as a sum of three terms:
\begin{align*}
    e^{\frac{\phi_{\epsilon}}{h}}h^2\mathcal{L}^{*}_{A,q} e^{-\frac{\phi_{\epsilon}}{h}}v = P_1v + P_2v + P_3v
\end{align*}
where,
\begin{align*}
& P_1 := -h^2\Delta - \nabla_x\phi_{\epsilon}\cdot\nabla_x\phi_{\epsilon} + h\partial_t \phi_{\epsilon}, \quad P_2:= 2h \nabla_x\phi_{\epsilon}\cdot \nabla + h \Delta\phi - h^2 \partial_t,\\
& P_3 :=  h^2\big( 2 A\cdot\nabla - 2A\cdot\frac{\nabla_x\phi}{h} \big) + h^2\Tilde{q}^{*}.
\end{align*}
Next, we will use the pseudodifferential operator techniques to shift the index by -1 in the above estimate \eqref{adjoint interior carleman estimate}. Let $\Tilde{\Omega} \subset \mathbb{R}^n$ a bounded open set such that $\overline{\Omega} \subset \Tilde{\Omega}$. Fix $w \in C^1([0,T];C^{\infty}_0(\Omega))$ satisfy $w(T,x) = 0$, using the composition of pseudodifferential operators, we have
\begin{align*}
   \langle hD\rangle^{-1}(P_1 + P_2)\langle hD\rangle w = (P_1 + P_2)w.
\end{align*}
For any $w \in C^1([0,T];C^{\infty}_0(\Omega))$, the estimate gives 
\begin{align*}
    \|(P_1 + P_2)\langle hD\rangle w\|_{L^2(0,T;H^{-1}_{scl}(\mathbb{R}^n))} &= \|\langle hD\rangle^{-1}(P_1 + P_2)\langle hD\rangle w \|_{L^2(0,T;L^2(\mathbb{R}^n))}\\ &= \|(P_1 + P_2)w\|_{L^2(0,T;L^2(\mathbb{R}^n))} \geq C \sqrt{\epsilon h}(\|w\| + \|hD(w)\|). 
\end{align*}
Consider 
\begin{align*}
     \|P_3\langle hD\rangle w\|_{L^2(0,T;H^{-1}_{scl}(\mathbb{R}^n))} \leq 2\bigg(h\|A\|_{\infty}\|w\| + h^2\|A\|_{\infty}\|\nabla w\| + h^2 \|\Tilde{q}^{*}\|_{\infty}\|w\|\bigg).
\end{align*}
Hence using the inequality as used in \eqref{eq:2.3} and choosing $\epsilon = C'h$ with fixed $C'\gg1$, we get
\begin{align*}
    \|e^{\frac{\phi_{\epsilon}}{h}}h^2\mathcal{L}^{*}_{A,q} e^{-\frac{\phi_{\epsilon}}{h}}\langle hD\rangle w\|_{L^2(0,T;H^{-1}_{scl}(\mathbb{R}^n))} \geq C h \|w\|_{L^2(0,T;H^{1}_{scl}(\mathbb{R}^n))}.
\end{align*}
Now, let $\chi \in C^{\infty}_0(\Tilde{\Omega})$ such that $\chi = 1$ in $\overline{\Omega}_1$ where $\overline{\Omega}\subset \Omega_1\subset\Tilde{\Omega}$. Taking $w = \chi\langle h D\rangle^{-1}v$ in the above estimate and using
\begin{align*}
    \|(1-\chi)\langle hD\rangle^{-1}v\|_{L^2(0,T;H^{1}_{scl}(\mathbb{R}^n))} = \mathcal{O}(h^{\infty})\|v\|_{L^2(0,T;L^2(\mathbb{R}^n))}.
\end{align*}
and 
\begin{align*}
    \|v\|_{L^2(0,T;L^2(\mathbb{R}^n))} &= \|\langle hD\rangle^{-1}v\|_{L^2(0,T;H^{1}_{scl}(\mathbb{R}^n))}\\
    &\leq \|w\|_{L^2(0,T;H^{1}_{scl}(\mathbb{R}^n))} + \mathcal{O}(h^{\infty})\|v\|_{L^2(0,T;L^2(\mathbb{R}^n))}.
\end{align*}
We obtain
\begin{align*}
     \|v\|_{L^2(0,T;L^2(\mathbb{R}^n))} \leq C h\| e^{\frac{\phi_{\epsilon}}{h}}\mathcal{L}^{*}_{A,q} e^{-\frac{\phi_{\epsilon}}{h}}v\|_{L^2(0,T;H^{-1}_{scl}(\mathbb{R}^n))}.
\end{align*}
 Now, using the fact that $e^{\frac{\epsilon \phi^2}{2h}}$ is uniformly bounded in $\Omega_T$, we conclude
 \begin{align*}
     \|v\|_{L^2(0,T;L^2(\mathbb{R}^n))} \leq C h\| e^{\frac{\phi}{h}}\mathcal{L}^{*}_{A,q} e^{-\frac{\phi}{h}}v\|_{L^2(0,T;H^{-1}_{scl}(\mathbb{R}^n))}.
\end{align*}
holds for $v \in C^1([0,T];C^{\infty}_0(\Omega))$ satisfying $v(T,x)=0$, where $h$ is small.\\
The above estimate, together with the Hahn-Banach theorem and the Riesz representation theorem gives us the solvability, the proof follows by \cite{sahoo2019partial}, \cite{senapati2021stability} and we get there exists a $u\in L^2(0, T; H^1(\Omega))\cap H^{1}(0, T; H^{-1}(\Omega))$ such that 
\begin{align*}
    e^{-\frac{\phi}{h}}h^2\mathcal{L}_{A,q}  e^{\frac{\phi}{h}}u = F, \quad  u(0,x) = 0, \ x\in \Omega, \quad
    h\|u\|_{L^2(0,T;H^1_{scl}(\Omega))}\leq C\|F\|_{L^2((0,T)\times\Omega)}.
\end{align*}
\end{proof}
\noindent From above proposition, there exists a $r(t,x;h) \in L^2(0,T;H^1(\Omega))\cap H^{1}(0,T;H^{-1}(\Omega))$ such that 
\begin{align*}
  e^{-\frac{1}{h}(\varphi+i\psi)\eta(t;h)}  h^2\mathcal{L}_{A,q}e^{\frac{1}{h}(\varphi+i\psi)\eta(t;h)}r(t,x;h) = - h^2 \mathcal{L}_{A,q} a+ h(\varphi+i\psi)\eta_t a= O(h^{7/5})
\end{align*}
and
\begin{align*}
    h\|r\|_{L^2(0,T;H^1_{scl}(\Omega))}\leq O(h^{7/5}).
\end{align*}
\begin{lemma}
 Let $x_0 \in \mathbb{R}^n \setminus \overline{ch(\Omega)}$ and $\varphi(x)$ be a limiting Carleman weight which is given by \eqref{eq:3.2}. Let $\psi(x)$ and  $\eta(t;h)$ is given by \eqref{eq:3.3} and \eqref{eq:3.4} respectively, then there exist $h_0>0$ and $r(t,x;h)$ such that
\begin{align*}
     h\|r\|_{L^2(0,T;H^1_{scl}(\Omega))}\leq O(h^{7/5}),
\end{align*}
and
\begin{align*}
    u = e^{\frac{(\varphi+i\psi)}{h}\eta(t;h)}(a(t,x) + r(t,x;h)),
\end{align*}
satisfying $u(0,x) = 0$ is a solution of $\mathcal{L}_{A,q}u=0$, when $h\leq h_0$, where $a(t,x) = m(t)e^{\Phi(t,x)}$, $\Phi(t,x)$ is a solution of \eqref{eq:3.5} and $m(t)$ is any non-vanishing real valued differentiable function.
\end{lemma}

\subsection{Construction of exponentially decaying solution}
The $L^2$-adjoint operator of $\mathcal{L}_{A,q}$ is as follows:
\begin{align*}
   \mathcal{L}^*_{A,q} = -\partial_t - \sum_{j=1}^{n}(\partial_j - A_j(t,x))^2 + \overline{q}.
\end{align*}
With $\varphi, \psi$ and $\eta$ as before, we can consider
\begin{align*}
    h^2 \mathcal{L}^*_{A,q}e^{\frac{1}{h}(-\varphi+i\psi)\eta}b = e^{\frac{1}{h}(-\varphi+i\psi)\eta}(h^2 \mathcal{L}^*_{A,q} - h(-\varphi+i\psi)\eta_t - h\eta L)b,
\end{align*}
where  $L$ is the transport operator given by
\begin{align*}
    L = (-\Delta\varphi + i\Delta\psi) + 2(-\nabla\varphi + i\nabla\psi)\cdot\nabla - 2A\cdot(-\nabla\varphi+i\nabla\psi).
\end{align*}
We wish to find $b$ satisfying $Lb = 0$ of the form $b = \widetilde{m}(t)e^{\widetilde{\Phi}(t,x)}$ for any non-vanishing real-valued differentiable function $\widetilde{m}(t)$.  In order to find such a function $b$, $\widetilde{\Phi}(t,x)$ must satisfy 
\begin{align}
\label{eq:47}
(-\Delta\varphi + i\Delta\psi) + 2(-\nabla\varphi + i\nabla\psi)\cdot\nabla\widetilde{\Phi} - 2A\cdot(-\nabla\varphi+i\nabla\psi) = 0.
\end{align}
For $- \varphi+i\psi$ and $\widetilde{m}(t)e^{\widetilde{\Phi}(t,x)}$, we obtain an approximate solution of the adjoint equation 
\begin{align*}
     h^2 \mathcal{L}^*_{A,q}\big(e^{\frac{1}{h}(-\varphi+i\psi)\eta}\widetilde{m}(t)e^{\widetilde{\Phi}(t,x)}\big) = e^{\frac{1}{h}(-\varphi+i\psi)\eta}(h^2 \mathcal{L}^*_{A,q} - h(-\varphi+i\psi)\eta_t)\widetilde{m}(t)e^{\widetilde{\Phi}(t,x)} = e^{\frac{-\varphi}{h}\eta}O(h^{7/5}).
\end{align*}
We can transform this approximate solution into the exact solution. The proof of this proposition can be established similarly as for Proposition \ref{Proposition 1}.
\begin{Prop} \label{Proposition 2}
Let $\phi$ be a limiting Carleman weight and $q\in L^{\infty}(\Omega_T)$. Then for small enough $h>0$ and for all $F\in L^2(\Omega_T)$ there exists a solution $v\in L^2(0,T;H^1(\Omega))\cap H^{1}(0,T;H^{-1}(\Omega))$ of
\begin{align*}
\left\{
	\begin{array}{ l l l }
  e^{\frac{\phi}{h}}h^2\mathcal{L}^{*}_{A,q} e^{-\frac{\phi}{h}}v = F, \ (t,x)\in\Omega_T\\
    v(T,x) = 0, \ x\in \Omega 
\end{array}
\right.
\end{align*}
satisfying the following estimate
\begin{align*}
    h\|v\|_{L^2(0,T;H^1_{scl}(\Omega))}\leq C\|F\|_{L^2(\Omega_T)}.
\end{align*}
\end{Prop}
\noindent From the above proposition \ref{Proposition 2}, we have the following lemma, which states the exact solution for the adjoint operator.
\begin{lemma}
 Let $x_0 \in \mathbb{R}^n \setminus \overline{ch(\Omega)}$ and $\varphi(x)$ be a limiting Carleman weight which is given by \eqref{eq:3.2}. Let $\psi(x)$ and  $\eta(t;h)$ is given by \eqref{eq:3.3} and \eqref{eq:3.4} respectively, then there exist $h_0>0$ and $\widetilde{r}(t,x;h)$ such that
\begin{align*}
     h\|\widetilde{r}\|_{L^2(0,T;H^1_{scl}(\Omega))}\leq O(h^{7/5}),
\end{align*}
and
\begin{align}
\label{eq:48}
    v = e^{\frac{(-\varphi+i\psi)}{h}\eta(t;h)}(b(t,x) + \widetilde{r}(t,x;h)),
\end{align}
satisfying $v(T,x) = 0$ is a solution of $\mathcal{L}^*_{A,q}v=0$, when $h\leq h_0$, where $b = \widetilde{m}(t)e^{\widetilde{\Phi}(t,x)}$, where $\widetilde{\Phi}(t,x)$ satisfies \eqref{eq:47} and $\widetilde{m}(t)$ is any non-vanishing real valued differentiable function.
\end{lemma}

\section{Proof of Theorem \ref{th:main theorem}}\label{Proof of main theorem}
As a next step to prove our main theorem, we devote this section to derive an integral identity that will play an essential role in the recovery of  $A$ and $q$. 

For $i = 1, 2$, let $A^{(i)}$ and  $q_i$ be as in the Theorem \ref{th:main theorem} and $u_i$  be the corresponding solution of IBVP for $\mathcal{L}_{A^{(i),q_i}}$ given in equation \eqref{eq:1.1}. More specifically, $u_i$, for $i = 1,2$, satisfy
\begin{equation} \label{eq:4.2}
  \begin{array}{cc}
  \left\{\begin{array}{ r l r}
    \mathcal{L}_{A^{(i),q_i}}u_i(t,x) &= 0, & (t,x)\in \Omega_T \\
	 u_i(0,x) &= 0,  & x \in \Omega   \\
	 u_i(t,x) &= f(t,x), & (t,x) \in \partial\Omega_T.
	\end{array}\right.
  \end{array}
 \end{equation}
To simplify the coming calculation, let us denote by
\begin{equation}\label{eq:4.3}
	\begin{array}{ r l }
 u(t,x) &:= u_1(t,x) - u_2(t,x),\\
     A(t,x) &:= A^{(1)}(t,x) - A^{(2)}(t,x), \\
	\Tilde{q}_i(t,x) &:= - \nabla_x\cdot A^{(i)}(t,x) - |A^{(i)}(t,x)|^2 + q_i(t,x),   \\
	\Tilde{q}(t,x) &:= \Tilde{q}_2(t,x) - \Tilde{q}_1(t,x).
	\end{array}
 \end{equation}
Then $u$ solves the following IBVP with zero initial and boundary conditions:
\begin{equation}\label{eq:4.4}
  \left\{
	\begin{array}{ r l c }
    \mathcal{L}_{A^{(1)},q_1}u(t,x) &= 2A\cdot\nabla_x u_2 + \Tilde{q}u_2, \quad &(t,x)\in \Omega_T \\
	 u(0,x) &= 0,  \quad &x \in \Omega   \\
	 u(t,x) &= 0, \quad &(t,x) \in \partial\Omega_T.
	\end{array}
	\right.
 \end{equation}
Assume that $v_1(t,x)$ is a solution to the following equation for the operator adjoint to $ \mathcal{L}_{A^{(1)},q_1}$ 
\begin{align} \label{eq:52}
   \left\{\begin{array}{rlc}
    \mathcal{L}^*_{A^{(1)},q_1} v_1(t,x) &= 0, &(t,x)\in \Omega_T  \\
       v_1(T, x) &=0,  & x \in \O.
   \end{array}\right.
\end{align}
Now multiplying \eqref{eq:4.4} by $\overline{v_1}(t,x)$ and then integrating over $\Omega_T$ yields 
\begin{align}
 \int_{\Omega_T} (2A\cdot\nabla_x u_2 + \Tilde{q}u_2)\overline{v_1}\ dx dt &= \int_{\Omega_T}  \mathcal{L}_{A^{(1)},q_1}u(t,x) \overline{v_1}(t,x)\ dx dt \nonumber\\
 & = \int_{\Omega_T} u(t,x) \overline{ \mathcal{L}^*_{A^{(1)},q_1} v_1(t,x)}\ dx dt - \int_{\partial\Omega_T} \partial_{\nu}u(t,x) \overline{v_1}(t,x)\ dS_x dt \nonumber\\ 
 &\quad +  \int_{\Omega} u(T,x) \overline{v_1}(T,x)\ dx,\label{eq:4.6}
\end{align}
where, we have used $u|_{\partial\Omega_T} = 0, u|_{t=0} = 0$ and $A^{(1)}(t,x) = A^{(2)}(t,x)$ on $\partial\Omega_T$. Since $v_1$ satisfy equation \eqref{eq:52}, we have
 \begin{align}
 \label{eq:55}
   \int_{\Omega_T} (2A\cdot\nabla_x u_2 + \Tilde{q}u_2)\overline{v_1}\ dx dt =   - \int_{\partial\Omega_T} \partial_{\nu}u(t,x) \overline{v_1}(t,x)\ dS_x dt
 \end{align}
Now using the equations \eqref{eq:1.2} and \eqref{eq:1.3}, the left-hand side of above relation satisfy 
$$    \langle(\Lambda_{A^{(1)},q_1} - \Lambda_{A^{(2)},q_2})(f), v_1|_{\partial\Omega_T}\rangle = - \int_{\Omega_T} (2A\cdot\nabla_x u_2 + \Tilde{q}u_2)\overline{v_1}\ dx dt,  $$
which implies
\begin{align}
    \langle(\Lambda_{A^{(1)},q_1} - \Lambda_{A^{(2)},q_2})(f), v_1|_{\partial\Omega_T}\rangle & =  \int_{\partial\Omega_T} \partial_{\nu}u(t,x) \overline{v_1}(t,x)\ dS_x dt \nonumber\\
    \label{eq:56}
    (\Lambda_{A^{(1)},q_1} - \Lambda_{A^{(2)},q_2})(f)|_{\partial\Omega_T} & = \partial_{\nu}u|_{\partial\Omega_T}.
\end{align}
For an $\epsilon_0>0$, define
\begin{align*}
    \partial\Omega_{+,\epsilon_0} := \{x \in \partial\Omega : \nabla_x\phi\cdot\nu(x) \geq \epsilon_0\},\\
    \partial\Omega_{-,\epsilon_0} := \{x \in \partial\Omega : \nabla_x\phi\cdot\nu(x) < \epsilon_0\}.
\end{align*}
Note $\partial\Omega_{+,\epsilon_0} \subset B(x_0), F(x_0)\subset \partial\Omega_{-,\epsilon_0}$. Choose $\epsilon_0$ small enough such that  $\partial\Omega_{-,\epsilon_0} \subset \widetilde{F}$, then equation \eqref{eq:1.5} implies
\begin{align*}
     \Lambda_{A^{(1)},q_1}f(x) =  \Lambda_{A^{(2)},q_2}f(x), \quad \forall x \in \partial\Omega_{-,\epsilon_0}, \quad\forall f \in \mathcal{K}_0.
\end{align*}
Hence from \eqref{eq:56}, we see
\begin{align*}
    \partial_{\nu}u_1 = \partial_{\nu}u_2 \quad \text{in} \quad (0,T)\times\partial\Omega_{-,\epsilon_0}
\end{align*}
and 
\begin{align*}
    supp( \partial_{\nu}u)|_{\partial\Omega_T} \subset (0,T)\times\partial\Omega_{+,\epsilon_0} \subset (0,T)\times\partial\Omega_+.
\end{align*}
Thus from \eqref{eq:55}, we have
\begin{align}
\label{eq:4.7}
    \int_{\Omega_T} (2A\cdot\nabla_x u_2 + \Tilde{q}u_2)\overline{v_1}\ dx dt = - \int_{(0,T)\times\partial\Omega_{+,\epsilon_0}}  \partial_{\nu}u(t,x) \overline{v_1}(t,x)\ dS_x dt
\end{align}

\begin{lemma}
\label{Lemma 3}
We have the following inequality
\begin{align*}
\bigg| \int_{(0,T)\times\partial\Omega_{+,\epsilon_0}}  \partial_{\nu}u(t,x) \overline{v_1}(t,x)\ dS_x dt\bigg| \leq Ch^{-1/2}. 
\end{align*}
\end{lemma}
\begin{proof} Recall $\displaystyle \overline{v_1}(t,x) = e^{\frac{-\varphi-i\psi}{h}\eta(t;h)}(\overline{b_1(t,x)+\widetilde{r}_1(t,x;h)})$ (see \eqref{eq:48}) and consider
\begin{align*}
    \bigg| \int_{(0,T)\times\partial\Omega_{+,\epsilon_0}}  \partial_{\nu}u(t,x) \overline{v_1}(t,x)\ dS_x dt\bigg| &\leq \int_{(0,T)\times\partial\Omega_{+,\epsilon_0}} | \partial_{\nu}u(t,x)| e^{-\frac{\varphi\eta}{h}}|(b_1+\widetilde{r}_1)| \ dS_xdt\\
     & \leq \|b_1+\widetilde{r}_1\|_{L^2((0,T)\times \partial\Omega_{+,\epsilon_0})}\bigg( \int_{(0,T)\times\partial\Omega_{+,\epsilon_0}} | \partial_{\nu}u(t,x)|^2 e^{-\frac{2\varphi\eta}{h}}dS_xdt\bigg)^{1/2}.
\end{align*}
That is, 
\begin{align*}
    &\bigg| \int_{(0,T)\times\partial\Omega_{+,\epsilon_0}}  \partial_{\nu}u(t,x) \overline{v_1}(t,x)\ dS_x dt\bigg|^2 \\
    &\qquad \qquad \qquad \leq \|b_1+\widetilde{r}_1\|^2_{L^2((0,T)\times \partial\Omega_{+,\epsilon_0})}\bigg( \int_{(0,T)\times\partial\Omega_{+,\epsilon_0}} | \partial_{\nu}u(t,x)|^2 e^{-\frac{2\varphi\eta}{h}}\ dS_x dt\bigg)\\
    & \qquad \qquad \qquad  \leq \|b_1+\widetilde{r}_1\|^2_{L^2((0,T)\times\partial\Omega_{+,\epsilon_0})} \int_{(0,T)\times\partial\Omega_{+,\epsilon_0}} \frac{\eta\varphi^{'} \cdot \nu}{\epsilon_0}| \partial_{\nu}u(t,x)|^2 e^{-\frac{2\varphi\eta}{h}}\ dS_x dt.
\end{align*}
By using Carleman estimate \eqref{eq:2.32}, we get
\begin{align*}
     \bigg| \int_{(0,T)\times\partial\Omega_{+,\epsilon_0}}  \partial_{\nu}u(t,x) \overline{v_1}(t,x)\ dS_x dt\bigg|^2 &\leq  \frac{C_0 h}{\epsilon_0}\|b_1+\widetilde{r}_1\|^2_{L^2((0,T)\times \partial\Omega_{+,\epsilon_0})}\|e^{-\frac{\varphi\eta}{h}}\mathcal{L}_{A^{(1)},q_1}u\|^2.
\end{align*}
The right-hand side of the above estimate can be further simplified from \eqref{eq:4.4} with expression for $\displaystyle u_2(t,x) = e^{\frac{(\varphi+i\psi)}{h}\eta(t;h)}(m_2(t)e^{\Phi_2(t,x)} + r_2(t,x;h))$, and by a simple calculation, we get 
\begin{align*}
    \bigg| \int_{(0,T)\times\partial\Omega_{+,\epsilon_0}}  \partial_{\nu}u(t,x) \overline{v_1}(t,x)\ dS_x dt\bigg| \leq Ch^{-1/2}.
\end{align*}
\end{proof}
\noindent With the inequality proved in the above Lemma, 
 the left-hand side of equation \eqref{eq:4.7} can be estimated as follows 
\begin{align*}
  \bigg|\int_{\Omega_T} (2A\cdot\nabla_x u_2 + \Tilde{q}u_2)\overline{v_1}\ dx dt \bigg|  \leq Ch^{-1/2}
\end{align*}
Next, multiplying above inequality by $h^{3/5}$ and taking $h \rightarrow 0$, we obtain
\begin{align*}
    h^{3/5}\int_{\Omega_T} (2A\cdot\nabla_x u_2 + \Tilde{q}u_2)\overline{v_1}\ dx dt = 0 \hspace{0.5cm} \text{as} \hspace{0.5cm} h\rightarrow 0
\end{align*}
Putting the expressions of $\displaystyle u_2 = e^{\frac{(\varphi+i\psi)}{h}\eta(t;h)}(m_2(t)e^{\Phi_2} + r_2)$ and $\displaystyle \overline{v_1} =  e^{\frac{(-\varphi-i\psi)}{h}\eta(t;h)}\overline{(\widetilde{m}_1(t)e^{\widetilde{\Phi}_1} + \widetilde{r}_1)}$ in above inequality and by simplifying it further we achieve
\begin{align}
\label{eq:4.8}
 \underset{h\rightarrow 0}{\lim} \int_{\Omega_T} h^{-2/5}(A^{(1)} - A^{(2)})\cdot (\nabla \varphi + i \nabla\psi) \eta(t;h) \widetilde{m}_1(t) m_2(t) e^{\overline{\widetilde{\Phi}_1}+\Phi_2}\ dx dt = 0.
\end{align}
Since $\Phi_2$ and $\widetilde{\Phi}_1$ are solutions of the following equations (see \eqref{eq:3.5} and \eqref{eq:47}):
\begin{align}
\label{eq:4.9}
     2(\nabla\varphi + i\nabla\psi)\cdot\nabla\Phi_2+(\Delta\varphi + i\Delta\psi) + 2A^{(2)}\cdot(\nabla\varphi+i\nabla\psi) = 0.\\
     \label{eq:4.10}
    2(\nabla\varphi - i\nabla\psi)\cdot\nabla\widetilde{\Phi}_1+(\Delta\varphi - i\Delta\psi) - 2A^{(1)}\cdot(\nabla\varphi-i\nabla\psi) = 0.
\end{align}
Adding the complex conjugate of \eqref{eq:4.10} with \eqref{eq:4.9}, we obtain
\begin{align}
& \qquad  2(\nabla\varphi + i\nabla\psi)\cdot(\nabla\Phi_2+\nabla\overline{\widetilde{\Phi}_1})+2(\Delta\varphi + i\Delta\psi) + 2(A^{(2)}-A^{(1)})\cdot(\nabla\varphi+i\nabla\psi) = 0.\nonumber\\
&\Longrightarrow (\Delta\varphi + i\Delta\psi) + (\nabla\varphi + i\nabla\psi)\cdot (\nabla\Phi_2+\nabla\overline{\widetilde{\Phi}_1} + A^{(2)}-A^{(1)}) = 0.\nonumber\\
&\Longrightarrow(\nabla + A^{(2)}-A^{(1)}) \cdot ((\nabla\varphi + i\nabla\psi)  e^{\overline{\widetilde{\Phi}_1}+\Phi_2}) = 0.\nonumber\\
\label{A_1-A_2 difference}
&\Longrightarrow \nabla \cdot ((\nabla\varphi + i\nabla\psi)  e^{\overline{\widetilde{\Phi}_1}+\Phi_2}) = (A^{(1)}-A^{(2)}) \cdot ((\nabla\varphi + i\nabla\psi)  e^{\overline{\widetilde{\Phi}_1}+\Phi_2}).
\end{align}
Substituting this in \eqref{eq:4.8} implies 
\begin{align*}
    \underset{h\rightarrow 0}{\lim} \int_{\Omega_T} h^{-2/5}\nabla \cdot \big((\nabla\varphi + i\nabla\psi)e^{\overline{\widetilde{\Phi}_1}+\Phi_2}\big)\eta(t;h) \widetilde{m}_1(t) m_2(t)\ dx dt = 0.
\end{align*}
Using $\displaystyle \underset{h\rightarrow 0}{\lim} \frac{\eta(t,h)}{h^{2/5}} = (T-t)^2$ and  varying $\widetilde{m}_1(t)$, $m_2(t)$ leads to
\begin{align}
\label{eq:4.11}
    \int_{\Omega} \nabla \cdot \big((\nabla\varphi + i\nabla\psi) e^{\overline{\widetilde{\Phi}_1}+\Phi_2}\big)\ dx  = 0
\end{align}
holds for each $t\in(0,T)$.\\
We may replace $e^{\Phi_2}$ by $e^{\Phi_2}g(x)$ in the expression for $u_2$, if $g$ satisfies $(\nabla\varphi + i\nabla\psi)\cdot\nabla g = 0$, therefore \eqref{eq:4.11} can be replace by
\begin{align}
    \label{eq:4.12}
    \int_{\Omega} g(x)\nabla \cdot \big((\nabla\varphi + i\nabla\psi) e^{\overline{\widetilde{\Phi}_1}+\Phi_2}\big)\ dx  = 0
\end{align}
Now by proceeding in a similar way as in \cite{ferreira2007determining}, we get required relation $dA^{(1)}(t,x) = dA^{(2)}(t,x)$ for all $x\in\Omega$ and for each $t\in(0,T)$.


\subsection{Recovery of density coefficient}
In this subsection, we will prove the uniqueness of the density coefficient  $q$.
We have already shown above that 
\begin{align}\label{eq: dA}
    dA^{(1)} = dA^{(2)}  \quad \forall x\in \Omega \ \mbox{ and for each}\ t\in(0,T).
\end{align}
Since $\Omega$ is simply connected, we get $A^{(2)}(t,x) - A^{(1)}(t,x) = \nabla_x\Psi(t,x)$ where $\Psi(t,x)\in W_0^{2,\infty}(\Omega_T)$. Now if we replace the pair $(A^{(1)},q_1)$ by $(A^{(3)},q_3)$ where $A^{(3)}(t,x) = A^{(1)}(t,x) + \nabla_x\Psi(t,x)$ and $q_3(t,x) = q_1(t,x) + \partial_t\Psi(t,x)$ then using the fact that $\Psi(t,x)\in W_0^{2,\infty}(\Omega_T)$ and Equation \eqref{eq:1.5}, we get $\Lambda_{A^{(3)},q_3} = \Lambda_{A^{(2)},q_2}$. By using the previous arguments and from Equation \eqref{eq:4.7}, we get
\begin{align*}
     \int_{\Omega_T} (q_3-q_2)u_2\overline{v_3}\ dx dt = - \int_{(0,T)\times\partial\Omega_{+,\epsilon_0}}  \partial_{\nu}(u_3 - u_2)(t,x) \overline{v_3}(t,x)\ dS_x dt
\end{align*}
where $u_3$ satisfies \eqref{eq:4.2}, $v_3$ satisfies the adjoint system \eqref{eq:52} with coefficients $A^{(3)},q_3$.
Now using the expression for $\displaystyle u_2 = e^{\frac{(\varphi+i\psi)}{h}\eta(t;h)}(m_2(t)e^{\Phi_2} + r_2)$ and $\displaystyle \overline{v_3} =  e^{\frac{(-\varphi-i\psi)}{h}\eta(t;h)}\overline{(\widetilde{m}_3(t)e^{\widetilde{\Phi}_3} + \widetilde{r}_3)}$, taking $h\rightarrow 0$, we get
 \begin{align*}
     \underset{h\rightarrow 0}{\lim}\int_{\Omega_T} (q_3-q_2) m_2(t)\widetilde{m}_3(t)e^{\overline{\widetilde{\Phi}_3(t,x)}+\Phi_2(t,x)}\ dxdt = 0.
 \end{align*}
For any $g$ satisfying $(\nabla\varphi + i\nabla\psi)\cdot\nabla g = 0$, we can replace $e^{\Phi_2}$ by $ge^{\Phi_2}$, then the above equation becomes
\begin{align}
\label{eq:66}
    \underset{h\rightarrow 0}{\lim}\int_{\Omega_T} (q_3-q_2) m_2(t)\widetilde{m}_3(t)g(x)e^{\overline{\widetilde{\Phi}_3(t,x)}+\Phi_2(t,x)}\ dxdt = 0. 
\end{align}
Since $A^{(2)}(t,x) = A^{(3)}(t,x)$, from \eqref{A_1-A_2 difference} we get $\overline{\widetilde{\Phi}_3}+\Phi_2$ is independent of time. Let us take 
\begin{align}
\label{eq:67}
    H(x) = \int_{0}^{T}(q_3-q_2)m_2(t)\widetilde{m}_3(t)\ dt,
\end{align}
therefore, we have from Equation \eqref{eq:66}
\begin{align*}
    \int_{\Omega}  H(x)g(x)e^{\overline{\widetilde{\Phi}_3}+\Phi_2}\ dx = 0.
\end{align*}
Following the proof of the recovering of potential in section 6 of \cite{kenig2007determining}, we can immediately derive that $H(x) = 0$. From \eqref{eq:67}, since $m_2(t)$ and $\widetilde{m}_3(t)$ are arbitrary, we obtain $q_2(t,x) = q_1(t,x) + \partial_t\Psi(t,x) $.
\section{Proof of Corollary  \ref{Cor: Corollary 2} and  \ref{Cor: Corollary 3}} \label{Proof of Corollary 2 and 3}
\subsection{Proof of Corollary \ref{Cor: Corollary 2}}
Using the fact that $A^{(1)}, A^{(2)} $ are independent of time, from \eqref{eq: dA} we get
\begin{align}
   A^{(2)}(x) - A^{(1)}(x) = \nabla_x\Psi(x) 
\end{align}
where $\Psi(x)\in W_0^{2,\infty}(\Omega)$. We have $\Psi$ independent of time, i.e., $\partial_t\Psi = 0$. By applying the same arguments as above for recovery of density coefficient, we get $q_1(t,x) = q_2(t,x)$ for $(t,x)\in \Omega_T$.

\subsection{Proof of Corollary \ref{Cor: Corollary 3}}
We have $\Psi \in  W_0^{2,\infty}(\Omega_T)$ such that
\begin{align}\label{difference of A}
   A^{(2)}(t,x) - A^{(1)}(t,x) = \nabla_x\Psi(t,x) 
\end{align}
Now using the condition \eqref{eq: divergence condition} and Equation \eqref{difference of A}, we get
\begin{align}
    \left\{
	\begin{array}{ r l l }
   \Delta_x\Psi(t,x) &= 0, \quad &(t,x)\in \Omega_T \\
	 \Psi(t,x) &= 0,   \quad & (t,x) \in \partial\Omega_T.
	\end{array}
	\right. 
\end{align}
Using the well-posedness of the above equation, we have $ \Psi(t,x) = 0$ for $(t,x)\in \Omega_T$. Thus from Equation \eqref{difference of A}, we get $A^{(2)}(t,x) = A^{(1)}(t,x)$ for $(t,x)\in \Omega_T$ and repeating the arguments for recovery of density coefficient, we get $q_1(t,x) = q_2(t,x), (t,x)\in \Omega_T$.
\section{Acknowledgements}\label{sec:acknowledge}
I am thankful to Prof. Manmohan Vashisth and Prof. Rohit Kumar Mishra for their helpful discussions and many valuable suggestions that helped me improve the paper. The University Grant Commission is gratefully acknowledged for the Ph.D. scholarship.

\bibliographystyle{plain}
\bibliography{Sample}

\begin{thebibliography}{10}

\bibitem{ibtissem2017}
I.~B. Aicha.
\newblock Stability estimate for an inverse problem for the schrödinger equation in a magnetic field with time-dependent coefficient.
\newblock {\em J. Math. Phys.}, 58(7), 2017.

\bibitem{bellassoued2021stably}
M.~Bellassoued and O.~B. Fraj.
\newblock Stably determining time-dependent convection-diffusion coefficients from a partial dirichlet-to-neumann map.
\newblock {\em Inverse Problems}, 37(4):045011, 2021.

\bibitem{bellassoued2008stability}
M.~Bellassoued, D.~Jellali, and M.~Yamamoto.
\newblock Stability estimate for the hyperbolic inverse boundary value problem by local dirichlet-to-neumann map.
\newblock {\em Journal of mathematical analysis and applications}, 343(2):1036--1046, 2008.

\bibitem{bellassoued2016inverse}
M.~Bellassoued, Y.~Kian, and E.~Soccorsi.
\newblock An inverse stability result for non-compactly supported potentials by one arbitrary lateral neumann observation.
\newblock {\em Journal of Differential Equations}, 260(10):7535--7562, 2016.

\bibitem{bellassoued2020stability}
M.~Bellassoued and I.~Rassas.
\newblock Stability estimate for an inverse problem of the convection-diffusion equation.
\newblock {\em Journal of Inverse and Ill-posed Problems}, 28(1):71--92, 2020.

\bibitem{cannon1963determination}
J.~R. Cannon.
\newblock Determination of an unknown coefficient in a parabolic differential equation.
\newblock {\em Duke Math. J.}, 30(2):313--323, 1963.

\bibitem{caro2018determination}
P.~Caro and Y.~Kian.
\newblock Determination of convection terms and quasi-linearities appearing in diffusion equations.
\newblock {\em arXiv preprint arXiv:1812.08495}, 2018.

\bibitem{cheng2000global}
J.~Cheng and M.~Yamamoto.
\newblock The global uniqueness for determining two convection coefficients from dirichlet to neumann map in two dimensions.
\newblock {\em Inverse Problems}, 16(3):L25, 2000.

\bibitem{cheng2002identification}
J.~Cheng and M.~Yamamoto.
\newblock Identification of convection term in a parabolic equation with a single measurement.
\newblock {\em Nonlinear Analysis-Theory Methods and Applications}, 50(2):163--172, 2002.

\bibitem{cheng2004determination}
J.~Cheng and M.~Yamamoto.
\newblock Determination of two convection coefficients from dirichlet to neumann map in the two-dimensional case.
\newblock {\em SIAM journal on mathematical analysis}, 35(6):1371--1393, 2004.

\bibitem{choulli1991abstract}
M.~Choulli.
\newblock An abstract inverse problem.
\newblock {\em Journal of Applied Mathematics and Stochastic Analysis}, 4(2):117--128, 1991.

\bibitem{mchoulli1991abstract}
M.~Choulli.
\newblock Abstract inverse problem and application.
\newblock {\em Journal of Mathematical Analysis and Applications}, 160(1):190--202, 1991.

\bibitem{choulli2009introduction}
M.~Choulli.
\newblock {\em Une introduction aux probl{\v{c}}mes inverses elliptiques et paraboliques}, volume~65.
\newblock Springer Science \& Business Media, 2009.

\bibitem{choulli2012stability}
M.~Choulli and Y.~Kian.
\newblock Stability of the determination of a time-dependent coefficient in parabolic equations.
\newblock {\em Mathematical Control and Related Fields}, 3:143--160, 2013.

\bibitem{choulli2018logarithmic}
M.~Choulli and Y.~Kian.
\newblock Logarithmic stability in determining the time-dependent zero order coefficient in a parabolic equation from a partial dirichlet-to-neumann map. application to the determination of a nonlinear term.
\newblock {\em Journal de Math{\'e}matiques Pures et Appliqu{\'e}es}, 114:235--261, 2018.

\bibitem{deng2008identifying}
Z.~C. Deng, J.~N. Yu, and L.~Yang.
\newblock Identifying the coefficient of first-order in parabolic equation from final measurement data.
\newblock {\em Mathematics and Computers in Simulation}, 77(4):421--435, 2008.

\bibitem{fan2021determining}
J.~Fan and Z.~Duan.
\newblock Determining a potential of the parabolic equation from partial boundary measurements.
\newblock {\em Inverse Problems}, 37(9):095001, 2021.

\bibitem{ferreira2007determining}
D.~D.~S. Ferreira, C.~E. Kenig, J.~Sj{\"o}strand, and G.~Uhlmann.
\newblock Determining a magnetic schr{\"o}dinger operator from partial cauchy data.
\newblock {\em Communications in mathematical physics}, 271:467--488, 2007.

\bibitem{gaitan2013stability}
P.~Gaitan and Y.~Kian.
\newblock A stability result for a time-dependent potential in a cylindrical domain.
\newblock {\em Inverse Problems}, 29(6):065006, 2013.

\bibitem{hoffmann2005generic}
K.~H. Hoffmann and M.~Yamamoto.
\newblock Generic uniqueness and stability in some inverse parabolic problem.
\newblock In {\em Inverse Problems in Mathematical Physics: Proceedings of The Lapland Conference on Inverse Problems Held at Saariselk{\"a}, Finland, 14--20 June 1992}, pages 49--54. Springer, 2005.

\bibitem{isakov1991completeness}
V.~Isakov.
\newblock Completeness of products of solutions and some inverse problems for pde.
\newblock {\em Journal of differential equations}, 92(2):305--316, 1991.

\bibitem{isakov1993uniqueness}
V.~Isakov.
\newblock On uniqueness in inverse problems for semilinear parabolic equations.
\newblock {\em Archive for Rational Mechanics and Analysis}, 124:1--12, 1993.

\bibitem{kachalov2001inverse}
A.~Kachalov, Y.~Kurylev, and M.~Lassas.
\newblock {\em Inverse boundary spectral problems}.
\newblock CRC Press, 2001.

\bibitem{kamynin2011unique}
V.~L. Kamynin.
\newblock Unique solvability of the inverse problem of determination of the leading coefficient in a parabolic equation.
\newblock {\em Differential equations}, 47:91--101, 2011.

\bibitem{kamynin2010two}
V.~L. Kamynin and A.~B. Kostin.
\newblock Two inverse problems of finding a coefficient in a parabolic equation.
\newblock {\em Differential Equations}, 46:375--386, 2010.

\bibitem{kenig2007determining}
C.~E. Kenig, J.~Sj{\"o}strand, and G.~Uhlmann.
\newblock The calder$\acute{o}$n problem with partial data.
\newblock {\em Ann. of Math.}, 165:567–--591, 2007.

\bibitem{kian2016recovery}
Y.~Kian.
\newblock Recovery of time-dependent damping coefficients and potentials appearing in wave equations from partial data.
\newblock {\em SIAM Journal on Mathematical Analysis}, 48(6):4021--4046, 2016.

\bibitem{kian2016stability}
Y.~Kian.
\newblock Stability in the determination of a time-dependent coefficient for wave equations from partial data.
\newblock {\em Journal of Mathematical Analysis and Applications}, 436(1):408--428, 2016.

\bibitem{kian2017unique}
Y.~Kian.
\newblock Unique determination of a time-dependent potential for wave equations from partial data.
\newblock In {\em Annales de l'Institut Henri Poincar{\'e} C, Analyse non lin{\'e}aire}, volume~34, pages 973--990. Elsevier, 2017.

\bibitem{oksanen2017}
Y.~Kian and L.~Oksanen.
\newblock {Recovery of Time-Dependent Coefficient on Riemannian Manifold for Hyperbolic Equations}.
\newblock {\em International Mathematics Research Notices}, 2019(16):5087--5126, 2017.

\bibitem{kian2014carleman}
Y.~Kian, Q.~P. Sang, and E.~Soccorsi.
\newblock A carleman estimate for infinite cylindrical quantum domains and the application to inverse problems.
\newblock {\em Inverse Problems}, 30(5):055016, 2014.

\bibitem{krishnan2020inverse}
V.~P. Krishnan and M.~Vashisth.
\newblock An inverse problem for the relativistic schr{\"o}dinger equation with partial boundary data.
\newblock {\em Applicable Analysis}, 99(11):1889--1909, 2020.

\bibitem{lions1968problèmes}
J.~L. Lions and E.~Magenes.
\newblock {\em Probl{\`e}mes aux limites non homog{\`e}nes et applications}, volume 1(17) of {\em Travaux et recherches math{\'e}matiques}.
\newblock Dunod Paris, 1968.

\bibitem{mishra2021determining}
R.~K. Mishra and M.~Vashisth.
\newblock Determining the time-dependent matrix potential in a wave equation from partial boundary data.
\newblock {\em Applicable Analysis}, 100(16):3492--3508, 2021.

\bibitem{mishra2022inverse}
R.~K. Mishra and M.~Vashisth.
\newblock Inverse problem for a time-dependent convection-diffusion equation in admissible geometries.
\newblock {\em arXiv preprint arXiv:2209.08780}, 2022.

\bibitem{murayama1981gel}
R.~Murayama.
\newblock The gel’fand-levitan theory and certain inverse problems for the parabolic equation.
\newblock {\em J. Fac. Sci. Univ. Tokyo}, 28:317--330, 1981.

\bibitem{NakamuraSasayama+2013+217+232}
G.~Nakamura and S.~Sasayama.
\newblock Inverse boundary value problem for the heat equation with discontinuous coefficients.
\newblock {\em Journal of Inverse and Ill-Posed Problems}, 21(2):217--232, 2013.

\bibitem{prilepko1992inverse}
A.~I. Prilepko and A.~B. Kostin.
\newblock Inverse problems of the determination of the coefficient in parabolic equations. i.
\newblock {\em Siberian Mathematical Journal}, 33(3):489--496, 1992.

\bibitem{rakesh1988uniqueness}
Rakesh and W.~W. Symes.
\newblock Uniqueness for an inverse problem for the wave equation: Inverse problem for the wave equation.
\newblock {\em Communications in partial differential equations}, 13(1):87--96, 1988.

\bibitem{sahoo2019partial}
S.~K. Sahoo and M.~Vashisth.
\newblock A partial data inverse problem for the convection-diffusion equation.
\newblock {\em Inverse Problems and Imaging}, 14(1):53--75, 2020.

\bibitem{senapati2020stability}
S.~Senapati.
\newblock Stability estimates for the relativistic schr{\"o}dinger equation from partial boundary data.
\newblock {\em Inverse Problems}, 37(1):015001, 2021.

\bibitem{senapati2021stability}
S.~Senapati and M.~Vashisth.
\newblock Stability estimate for a partial data inverse problem for the convection-diffusion equation.
\newblock {\em Evolution Equations and Control Theory}, 11(5):1681–1699, 2021.

\bibitem{sylvester1987global}
J.~Sylvester and G.~Uhlmann.
\newblock A global uniqueness theorem for an inverse boundary value problem.
\newblock {\em Annals of mathematics}, pages 153--169, 1987.

\end{thebibliography}
\end{document}